\documentclass[a4paper,11pt]{amsart}

\usepackage{amsmath, amsthm, amsfonts, amssymb, enumerate}
\usepackage[bookmarks=false]{hyperref} 
\usepackage{xcolor}
\usepackage{verbatim}
\usepackage[normalem]{ulem}

\title[A unique Cartan  result for Bernoulli actions of weakly amenable groups]{A unique Cartan subalgebra result for Bernoulli actions of weakly amenable groups}
\author{Changying Ding}

\address{Department of Mathematics, UCLA, Los Angeles, CA 90095, USA}
\email{cding@math.ucla.edu}

\newtheorem{thm}{Theorem}[section]
\newtheorem{prop}[thm]{Proposition}

\newtheorem{cor}[thm]{Corollary}
\newtheorem{lem}[thm]{Lemma}
\theoremstyle{definition}

\newtheorem{defn/lem}[thm]{Definition/Lemma}

\newcommand{\B}{{\mathbb B}}
\newcommand{\C}{{\mathbb C}}
\newcommand{\F}{{\mathbb F}}

\newcommand{\K}{{\mathbb K}}
\newcommand{\M}{{\mathbb M}}
\newcommand{\N}{{\mathbb N}}

\newcommand{\Q}{{\mathbb Q}}
\newcommand{\R}{{\mathbb R}}
\newcommand{\bS}{{\mathbb S}}
\newcommand{\T}{{\mathbb T}}

\newcommand{\X}{{\mathbb X}}
\newcommand{\Y}{{\mathbb Y}}
\newcommand{\Z}{{\mathbb Z}}
\newcommand{\cP}{{\mathcal P}}

\newcommand{\cG}{{\mathcal G}}
\newcommand{\cH}{{\mathcal H}}

\newcommand{\cK}{{\mathcal K}}

\newcommand{\cN}{{\mathcal N}}

\newcommand{\cU}{{\mathcal U}}

\newcommand{\cZ}{{\mathcal Z}}

\newcommand{\Ad}{\operatorname{Ad}}

\newcommand{\Aut}{\operatorname{Aut}}

\newcommand{\id}{\operatorname{id}}

\newcommand{\ot}{\otimes}
\newcommand{\mot}{\otimes_{\rm min}}
\newcommand{\ovt}{\, \overline{\otimes}\,}

\newcommand{\ds}{{\sharp\kern-.5pt\sharp}}

\newcommand{\actson}{{\, \curvearrowright \,}}

\newcommand{\reforder}[1]{}

\makeatletter
\DeclareRobustCommand\frownotimes{\mathbin{\mathpalette\frown@otimes\relax}}
\newcommand{\frown@otimes}[2]{%
  \vbox{
    \ialign{##\cr
      \hidewidth$\m@th#1{}_\frown$\kern-\scriptspace\hidewidth\cr
      \noalign{\nointerlineskip\kern-1pt}
      $\m@th#1\otimes$\cr
    }%
  }%
}
\makeatother

\begin{document}
\maketitle
\begin{abstract}
We show that if $\Gamma\actson (X^\Gamma,\mu^\Gamma)$ is a Bernoulli action of an i.c.c.\ nonamenable group $\Gamma$ which is weakly amenable with Cowling-Haagerup constant $1$, and $\Lambda\actson(Y,\nu)$ is a free ergodic p.m.p.\ algebraic action of a group $\Lambda$,
	then the isomorphism $L^\infty(X^\Gamma)\rtimes\Gamma\cong L^\infty(Y)\rtimes\Lambda$ implies that $L^\infty(X^\Gamma)$ and $L^\infty(Y)$ are unitarily conjugate.
This is obtained by showing a new rigidity result of non properly proximal groups 
	and combining it with a rigidity result of properly proximal groups from \cite{BoIoPe21}.
\end{abstract}

\section{Introduction}

The group measure space construction associates to every probability measure preserving (p.m.p.) action
$\Gamma\actson (X,\mu)$ of a countable group $\Gamma$, a finite von Neumann algebra $L^\infty(X)\rtimes\Gamma$ \cite{MuvN43}.
When the action is free and ergodic, $L^\infty(X)\rtimes\Gamma$ is a II$_1$ factor and $L^\infty(X)$ is a Cartan subalgebra.
With the discovery of Popa's deformation/rigidity theory \cite{Popa06, Po06B, Po06C}, 
	spectacular progress has been made in the classification and structural results of II$_1$ factors
	(see surveys \cite{Po07B, Va10, Io18}),
	and specifically, group measure space II$_1$ factors arising from Bernoulli actions 
	have been shown to possess extreme rigidity (see e.g. \cite{Po06B, Po06C, Po08, Ioa11, IoPoVa13}).
In fact, a conjecture of Popa states that if $\Gamma\actson(X^\Gamma, \mu^\Gamma)$ is
	a Bernoulli action of a nonamenable group $\Gamma$, then $L^\infty(X^\Gamma)\rtimes\Gamma$ has a unique Cartan subalgebra, up to unitary conjugacy \cite[Problem III]{Io18}.
Such unique Cartan subalgebra results play a crucial role in the classification of group measure space II$_1$ factors,
	as they allow one to reduce the classification of group measure space II$_1$ factors 
	to the classification of the orbit equivalence relations of the corresponding group actions \cite{Si55}.

Although this conjecture of Popa remains open in its full generality,	
significant progress has been made towards it during the last 15 years.
To mention a few breakthrough results, Ioana showed that $L^\infty(X^\Gamma)\rtimes\Gamma$ has a unique group measure space Cartan subalgebra
	when $\Gamma$ has Property (T) \cite{Ioa11}.
Subsequently, Ioana, Popa and Vaes showed the same conclusion holds if $\Gamma$ is a product group \cite{IoPoVa13}.
In another direction, Popa and Vaes proved that $L^\infty(Y)\rtimes\Gamma$ has a unique Cartan subalgebra for any free ergodic p.m.p.\ action $\Gamma\actson Y$,  
	when $\Gamma$ is weakly amenable and satisfies that either $\Gamma$ has positive first $\ell^2$-Betti number \cite{PoVa14I} or $\Gamma$ is biexact \cite{PoVa14II}.
Building upon \cite{PoVa14I}, Ioana showed that the same conclusion holds if $\Gamma$ is a free product group \cite{Ioa15}.
More recently, Boutonnet, Ioana and Peterson generalized \cite{PoVa14II} to groups that are weakly amenable and properly proximal \cite{BoIoPe21}.
	
In view of \cite{PoVa14I, PoVa14II},
 	to make further progress towards this conjecture,
	one may first consider the question of showing $L^\infty(X^\Gamma)\rtimes\Gamma$ has 
	a unique group measure space Cartan subalgebra for weakly amenable $\Gamma$,
	as suggested in \cite[Chapter V]{Bou14}.
Moreover, in light of \cite{BoIoPe21}, 
	it suffices to consider this question for groups that are non properly proximal.
Our main theorem is towards this direction.

\begin{thm}\label{thm: recover Bernoulli}
Let $\Gamma$ be a countable discrete nonamenable i.c.c. group, $(X_0, \mu_0)$ a diffuse standard probability space and $\Gamma\actson (X_0^\Gamma, \mu_0^\Gamma)=:(X,\mu)$ a Bernoulli action. 
Suppose $L^\infty(X)\rtimes\Gamma\cong L\Lambda$
	for some countable discrete group $\Lambda$.
If $\Gamma$ is non properly proximal and weakly amenable with Cowling-Haagerup constant $1$, 
	then $\Lambda\cong\Sigma\rtimes \Gamma$
	for some infinite abelian group $\Sigma$ such that $\Gamma\actson \Sigma$ by automorphisms,
	and $\Gamma\actson X$, $\Gamma\actson \widehat \Sigma$ are conjugate p.m.p.\ actions.
\end{thm}

Piecing Theorem \ref{thm: recover Bernoulli} and \cite{BoIoPe21} together, 
	we obtain the following result that does not involve proper proximality.
Recall that an algebraic action $\Lambda\actson Y$ is a homomorphism from $\Lambda$ to $\Aut(Y)$,
	with $Y$ a compact metrizable abelian group.

\begin{thm}\label{thm: semidirect product CMAP}
Let $\Gamma$ be a countable discrete i.c.c.\ nonamenable group, 
	$(X_0,\mu_0)$ a diffuse standard probability space
	and $\Gamma\actson^\sigma (X_0^\Gamma, \mu_0^\Gamma)=:(X,\mu)$ the Bernoulli action.
Suppose $\Gamma$ is weakly amenable with Cowling-Haagerup constant $1$.
For any group $\Lambda$ and any free ergodic p.m.p.\ algebraic action $\Lambda\actson Y$, 
	if $L^\infty(X)\rtimes\Gamma\cong L^\infty(Y)\rtimes\Lambda$, 
	then $L^\infty(X)$ and $L^\infty(Y)$ are unitarily conjugate.
\end{thm}
%Our approach is somewhat indirect: we consider separately when $\Gamma$ is properly proximal and non properly proximal in the sense of \cite{BoIoPe21}. 
%In both cases, (non) proper proximality provides us with additional rigidity to draw the conclusion.

%To further illustrate this approach, we pause our discussion of Theorem~\ref{thm: semidirect product CMAP}
%	and point out that the same approach allows one to answers \cite[Question V.1.2.]{Bou14},
%	which we record here.

As a side note, we point out that \cite[Question V.1.2]{Bou14} can be answered 
	by piecing together the properly proximal case and non properly proximal case,
	similar to Theorem~\ref{thm: semidirect product CMAP},
	only using existing results in the literature.
	
\begin{thm}\label{thm: strong rigidity}
Let $\Gamma$ be a nonamenable i.c.c.\ weakly amenable group and $\Gamma\actson X$ a free ergodic p.m.p.\ action.
Let $\Lambda$ be a nonamenable group and $\Lambda\actson [0,1]^\Lambda$ a Bernoulli action.
If $L^\infty(X)\rtimes\Gamma\cong L^\infty([0,1]^\Lambda)\rtimes\Lambda$,
	then $L^\infty(X)$ and $L^\infty([0,1]^\Lambda)$ are unitarily conjugate.
\end{thm}

Indeed, when $\Gamma$ is properly proximal, this is a special case of \cite[Theorem 1.5]{BoIoPe21}.
When $\Gamma$ is non properly proximal, \cite[Theorem 1.3]{Din22}
	shows that $\Gamma\actson X$ and $\Lambda\actson [0,1]^\Lambda$ are actually conjugate.
In both cases, we conclude that $L^\infty(X)$ and $L^\infty([0,1]^\Lambda)$ are unitarily conjugate.

%Back to our discussion of Theorem~\ref{thm: semidirect product CMAP}. 
%If we assume $\Gamma$ is in addition properly proximal, 
%	this is also a consequence \cite[Theorem 1.5]{BoIoPe21}.
%For the non properly proximal case, we establish the following rigidity result to complete the argument for
%	 Theorem~\ref{thm: semidirect product CMAP}.

Turning back to Theorem~\ref{thm: recover Bernoulli}, we remark that its proof follows the strategy of \cite{IoPoVa13} closely
	and exploits the rigidity of non properly proximal groups in the setting of Bernoulli actions by building upon results in \cite{Din22}.
We also note that Theorem~\ref{thm: recover Bernoulli} holds when $\Gamma$ has Property (T) or is a product group as well by \cite{IoPoVa13}.
%We also remark that the other half of the proof for Theorem~\ref{thm: semidirect product CMAP}, 
%	\cite[Theorem 1.5]{BoIoPe21}, follows the strategy of \cite{PoVa14I, PoVa14II},
%	which is quite different from the strategy of \cite{IoPoVa13}.
%	which we employ in Theorem~\ref{thm: recover Bernoulli}.
	
Some concrete examples beyond \cite{IoPoVa13} are covered by Theorem~\ref{thm: recover Bernoulli}. 
Groups considered in \cite[Example 0.9]{TD20} 
	are of the form $(\bigoplus_S \Z_2)\rtimes\F_n$, where $n\geq 2$ 
	and $S$ is an infinite set on which $\F_n$ acts on amenably. 
They are inner amenable \cite{TD20} and thus non properly proximal \cite{BoIoPe21},
	and $\Lambda_{\rm cb}((\bigoplus_S \Z_2)\rtimes\F_n)=\Lambda_{\rm cb}(\F_n)=1$
	by \cite[Proposition 3.2, Corollary 3.3]{OzPo10I}.
However, \cite{IoPoVa13} does not apply to these since the centralizer of any infinite subgroup of these groups is amenable \cite{TD20}.

Before outlining the proof, let us make some remarks on the assumptions of Theorem~\ref{thm: semidirect product CMAP}.
In addition to assuming weak amenability of $\Gamma$,
	we also assume $\Lambda_{\rm cb}(\Gamma)=1$ and $\Lambda\actson Y$ is an algebraic action.
The action $\Lambda\actson Y$ being algebraic allows us to realize $L^\infty(Y)\rtimes\Lambda$ as a group von Neumann algebra $L(\widehat Y\rtimes\Lambda)$,
	for which we may use the comultiplication map arising from a group instead of a group measure space \cite{PoVa10a}.
Using this comultiplication map enables us to analyze the relation between proper proximality and the comultiplication map in Section~\ref{sec: comultiplication}.
The reason to assume $\Lambda_{\rm cb}(\Gamma)=1$ is present in Section~\ref{sec: dichotomy}.
Roughly speaking, in the relative setting, the complete metric approximation property plays a role that is similar to local reflexivity.
This assumption is also present in \cite[Proposition 7.3]{Iso20} to avoid the same technical difficulty.

Let us finish the introduction by outlining the proof of Theorem~\ref{thm: recover Bernoulli} informally.
It is built on the recently developed notions of proper proximality and biexactness for von Neumann algebras \cite{DKEP22, DP22},
	and Popa's deformation/rigidity theory, in particular, the breakthrough work \cite{Ioa11, IoPoVa13}.
In view of \cite{IoPoVa13}, to prove Theorem~\ref{thm: recover Bernoulli}, 
	it suffices to show $\Delta(L\Gamma)$ can be unitarily conjugate into $L\Gamma\ovt L\Gamma$,
	where $\Delta: M:=L\Lambda\to L\Lambda\ovt L\Lambda$ 
	is the comultiplication map first considered in \cite{PoVa10a}.
First, we use the fact that $M\ovt M$ 
	is biexact relative to 
	$\{M\ovt L\Gamma, L\Gamma\ovt M\}$ \cite{BrOz08, DP22} 
	to conclude that $\Delta(L\Gamma)\subset M\ovt M$ is properly proximal relative to 
	$\{M\ovt L\Gamma, L\Gamma\ovt M\}$ \cite{DP22}.
In Proposition~\ref{prop: non prop prox to rel amen}, 
	we play this, together with malnormality of $\Gamma<\Z\wr\Gamma$,
	against the non properly proximality of $\Gamma$
	to conclude $\Delta(L\Gamma)$ can be unitarily conjugate into $M\ovt L\Gamma$.
Next, although $M\ovt L\Gamma$ has no relative biexactness,
	using the machinery from \cite{DP22}, we may still obtain certain proper proximality for 
	$\Delta(L\Gamma)\subset M\ovt L\Gamma$ by Proposition~\ref{prop: dichotomy}.
This allows us to exploit the non proper proximality of $\Gamma$ again in 
	Proposition~\ref{prop: non prop prox to rel amen} 
	and conclude that $\Delta(L\Gamma)$ may be unitarily conjugate into $L\Gamma\ovt L\Gamma$.

\textbf{Acknowledgment.}
I would like to thank Daniel Drimbe, Adrian Ioana, Jesse Peterson, Sorin Popa and Stefaan Vaes  for their useful comments.
I am also very grateful to  Jesse Peterson and Sorin Popa for their encouragement.

\section{Preliminaries}

\subsection{The small-at-infinity boundary and boundary pieces}\label{sec: boundary pieces}

In this section we recall the notion of the small-at-infinity boundary for von Neumann algebras developed in \cite{DKEP22, DP22}, 
	which is a noncommutative analogue of the corresponding boundary for groups
	introduced by Ozawa \cite{Oz04, BrOz08}.

Let $M$ be a finite von Neumann algebra.
An $M$-boundary piece $\X$ is a hereditary ${\rm C}^*$-subalgebra $\X\subset\B(L^2M)$
	such that $M\cap M(\X)\subset M$ and $JMJ\cap M(\X)\subset JMJ$ are weakly dense,
	and $\X\neq \{0\}$,
	where $M(\X)$ denotes the multiplier algebra of $\X$.
For convenience, we will always assume $\X\neq \{0\}$.
Given an $M$-boundary piece $\X$, define $\K_\X^L(M)\subset \B(L^2M)$ to be the $\|\cdot\|_{\infty,2}$ closure of $\B(L^2M)\X$,
	where $\|T\|_{\infty,2}=\sup_{\hat a\in (M)_1}\|T\hat a\|$
		and $(M)_1=\{a\in M\mid \|a\|\leq 1\}$.
Set $\K_\X(M)=\K_\X^L(M)^*\cap \K_\X^L(M)$,
	then $\K_\X(M)$ is a ${\rm C}^*$-subalgebra that contains $M$ and $JMJ$ in its multiplier algebra.
Put $\K^{\infty,1}_\X(M)=\overline{\K_\X(M)}^{_{\|\cdot\|_{\infty,1}}}\subset \B(L^2M)$, 
	where $\|T\|_{\infty,1}=\sup_{a,b\in (M)_1}\langle T\hat a, \hat b\rangle$,
	and the small-at-infinity boundary for $M$ relative to $\X$ is given by 
$$
\bS_\X(M)=\{T\in\B(L^2M)\mid [T,x]\in \K_\X^{\infty,1}(M),{\rm\ for\ any\ }x\in M'\}.
$$
When $\X=\K(L^2M)$, we omit $\X$ in the above notations.

The following instance of boundary pieces is extensively used in this paper.
Let $M$ be a finite von Neumann algebra and $\{P_i\}_{i=1}^n$ a family of von Neumann subalgebras.
Recall from \cite{DKEP22}
	that the $M$-boundary piece $\X$ associated with $\{P_i\}_{i=1}^n$ is the
	hereditary ${\rm C}^*$-subalgebra of $\B(L^2M)$ 
	generated by $\{x JyJ e_{P_i}\mid i=1,\dots, n, x,y\in M\}$.
If $[e_{P_i}, e_{P_j}]=0$ for $i,j=1,\dots, n$, we have $\X$ coincides with the hereditary ${\rm C}^*$-subalgebra generated by $\{x JyJ (\vee_{i=1}^n e_{P_i})\mid x,y\in M\}$ \cite[Lemma 3.2]{DKE22}.
When the family only contains one von Neumann subalgebra $P\subset M$,
	we usually denote the $M$-boundary piece by $\X_P$.

%In particular, when $M=L\Gamma$ for some discrete group $\Gamma$ and $P_i=L\Sigma_i$ for a family of subgroups $\{\Sigma_i\}_{i=1}^n$ of $\Gamma$, we have 
%	$\{P_{F(\cup_{i=1}^n \Sigma_i)F}\}_F$ forms approximate units of $\X_0$,
%	where $F$ ranges over all finite subsets of $\Gamma$,
%	$P_S$ denotes the orthogonal projection from $\ell^2\Gamma$ to ${\rm sp}\{\delta_g\mid g\in S\}$ for any subset $S\subset \Gamma$
%	and $\X_0$ is the hereditary ${\rm C}^*$-subalgebra generated by $\{xJyJ(\vee_{i=1}^n e_{P_i})\mid x,y\in C^*_r\Gamma\}$,
%	which is dense in $\K_\X^{\infty,1}(M)$ under $\|\cdot\|_{\infty,1}$
%	by the proof of \cite[Lemma 3.5]{Din22}.

\subsection{Biexactness and proper proximality}

With the notion of the small-at-infinity boundary in hand, 
	we recall proper proximality and biexactness for von Neumann algebras,
	which were introduced in \cite{DKEP22} and \cite{DP22}, respectively.
These are generalizations of their corresponding notions for groups introduced in \cite{BoIoPe21} and \cite{Oz04, BrOz08}, respectively.

Let $M$ be a von Neumann algebra and $\X$ an $M$-boundary piece.
Given a von Neumann subalgebra $N\subset pMp$ with a nonzero projection $p\in M$,
	we say $N$ is not properly proximal relative to $\X$ in $M$
	if there exists an $N$-central state $\varphi: p\bS_\X(M)p\to \C$ such that $\varphi_{\mid pMp}$ is normal.
Equivalently, $N$ is not properly proximal relative to $\X$ in $M$ if there exists some nonzero projection $z\in \cZ(N)$ and an $Nz$-bimodular u.c.p.\ map
	$\phi: z\bS_\X(M)z\to Nz$ such that $\phi_{\mid zMz}$ is normal.
When $\X$ is the $M$-boundary piece associated with 
	a family of von Neumann subalgebras $\{P_i\}_{i=1}^n$ of $M$,
	we say $N$ is not properly proximal relative to $\{P_i\}_{i=1}^n$ in $M$.

We say $M$ is biexact relative to $\X$ if there exist nets of u.c.p.\ maps 
$\phi_i: M\to \M_{n(i)}(\C)$ and $\psi_i: \M_{n(i)}(\C)\to \bS_\X(M)$
such that $\psi_i\circ \phi_i(x)\to x$ in the $M$-topology of $\bS_\X(M)$.
Here, by an equivalent characterization \cite[Lemma 3.4]{DP22}, 
	we say a net $\{x_i\}\subset \bS_\X(M)$ converging to $0$ in the $M$-topology if
	there exists a net of projections $p_i\in M$ such that $p_i\to 1$ strongly
	and $\|p_ix_ip_i\|\to 0$.
When $\X$ is the $M$-boundary piece associated with 
	a family of von Neumann subalgebras $\{P_i\}_{i=1}^n$ of $M$,
	we say $M$ is biexact relative to $\{P_i\}_{i=1}^n$.

These notions coincide with the corresponding notions of groups if we consider group von Neumann algebras:
a discrete group $\Gamma$ is properly proximal (resp. biexact) relative to a family of subgroups $\{\Lambda_i\}_{i=1}^n$ if and only if
$L\Gamma$ is properly proximal (resp. biexact) relative to $\{L\Lambda_i\}_{i=1}^n$
\cite{DKEP22, DP22}.

\subsection{Normal biduals}
A bidual characterization of proper proximality will be crucial to our arguments. 
In this section, we briefly recall necessary notions around normal biduals from \cite{DKEP22, DP22}.
	
Let $M$ be a finite von Neumann algebra and $A\subset \B(L^2M)$ a ${\rm C}^*$-subalgebra
	such that $M$ and $JMJ$ are contained in its multiplier algebra.
Denote by $A^\sharp_J$ the set of functionals $\varphi\in A^*$ such that for any $T\in A$,
$$M\times M\ni (a,b)\mapsto \varphi(aTb)\in \C,\ JMJ\times JMJ\ni(a,b)\mapsto \varphi(aTb)\in \C$$
are both separately normal.

The normal bidual of $A$, denoted by $A^{\sharp *}_J$, may be identified with a corner of $A^{**}$ 
	and be viewed as a von Neumann algebra.
Denote by $p_{\rm nor}\in M(A)^{**}$ the support projection of both identity representations of $M$ and $JMJ$. 
Equivalently, $p_{\rm nor}$ is the support projection of states in $M(A)^*$ that are normal when restricted to $M$ and $JMJ$.
Then $M(A) ^{\sharp *}_J$ may be identified with $p_{\rm nor} M(A)^{**} p_{\rm nor}$
	and $A ^{\sharp *}_J=q_A p_{\rm nor} M(A)^{**} p_{\rm nor}$,
	where $q_A\in M(A)^{**}$ is the identity of $A^{**}$ in $M(A)^{**}$.

Throughout the paper, we reserve the notation $p_{\rm nor}$ for the above projection and 
	set $\iota_{\rm nor}: \B(L^2M)\ni T\mapsto p_{\rm nor} \pi_u(T) p_{\rm nor}\in  \B(L^2M) ^{\sharp *}_J$,
	where $\pi_u: \B(L^2M)\to \B(L^2M)^{**}$ is the universal representation.
Notice that $\iota_{\rm nor}$ is no longer a $*$-homomorphism,
	but $\iota_{\rm nor}$ restricts to normal representations on $M$ and $JMJ$.

We consider the following bidual version of the small-at-infinity boundary for a von Neumann algebra $M$:
$$\tilde \bS(M)=\mid \{T\in \B(L^2M) ^{\sharp *}_J\mid [T, x]\in \K(M) ^{\sharp *}_J,{\ \rm for\ all\ }x\in JMJ\},$$
where we view $JMJ$ as in $\B(L^2M) ^{\sharp *}_J$ through the representation $\iota_{\rm nor}$.
	
By \cite[Lemma 8.5]{DKEP22}, we have that if a countable discrete group $\Gamma$ is non properly proximal, then there exists a $L\Gamma$-central state $\varphi:\tilde \bS(L\Gamma)\to \C$
	such that $\varphi_{\mid L\Gamma}=\tau$.

Next we collect a few lemmas for Section~\ref{sec: concentration of state}.

\begin{lem}\label{lem: identity}e
Let $\Gamma$ be a countable discrete group with a family of subgroups $\{\Sigma_i\}_{i=1}^n$.
Denote by $M=L\Gamma$, 
	$\Y$ the $M$-boundary piece associated with $\{L\Sigma_i\}_{i=1}^n$,
	$P_F$ the orthogonal projection from $\ell^2\Gamma$ onto
	$\overline{{\rm sp}\{\delta_g\mid g\in F(\cup_{i=1}^n \Sigma_i) F\}}$
	for a finite subset $F\subset \Gamma$,
	and $q_\Y$ the unit in $(\K_\Y(M)^\sharp_J)^*$.
Then $q_\Y=\lim_F \iota_{\rm nor} (P_F)$,
	where the limit is over finite subsets of $\Gamma $ 
	and is taken in $(\B(L^2M)^\sharp_J)^*$ with the weak$^*$ topology.
\end{lem}
\begin{proof}
Denote by $\Y_0$ the hereditary ${\rm C}^*$-subalgebra generated by 
	$\{xJyJ(\vee_{i=1}^n e_{L\Sigma_i})\mid x,y\in C^*_r\Gamma\}$.
One checks $\{P_F\mid F\subset \Gamma{\rm \ finite}\}$ forms an approximate unit for $\Y_0$.
Since $\iota_{\rm nor}(\Y_0)\subset (\K_\Y(M)^\sharp_J)^*$ is weak$^*$ dense \cite[Lemma 3.7]{DKE22},
	one has $\lim_F\iota_{\rm nor}(P_F)=q_\Y$.
\end{proof}

\begin{lem}\label{lem: commuting projections}
Let $\Gamma$ be a countable discrete group with a subgroup $\Lambda<\Gamma$ and a family of subgroups $\{\Sigma_i\}_{i=1}^n$.
Denote by $M=L\Gamma$ and $\Y$ the $M$-boundary piece associated with $\{L\Sigma_i\}_{i=1}^n$.
Set $q_\Y$ be the unit in $(\K_\Y(M)^\sharp_J)^*$ and $P$ the orthogonal projection from $\ell^2\Gamma$ to $\overline{{\rm sp}\{\delta_t\mid t\in \cup_{j=1}^d t_j \Lambda s_j\}}$ for a some $t_j, s_j\in \Gamma$.
Then $[q_\Y, \iota_{\rm nor}(P)]=0$.
\end{lem}
\begin{proof}
For a subset $S\subset \Gamma$, denote by $P_S$ the orthogonal projection onto $\overline{{\rm sp}\{\delta_t\mid t\in S\}}$.
Note that for any subgroup $\Gamma_0$, the projection $\pi_u(P_{\Gamma_0})$ commutes with $p_{\rm nor}$ \cite[Lemma 3.4]{Din22}.
Since $P_{t\Gamma_0 s}= \Ad(\lambda_t^*\rho_s)(P_{\Gamma_0})$, we have $\pi_u(P_{t\Gamma_0 s})$ commutes with $p_{\rm nor}$.
Moreover, for two commuting projections $Q_1, Q_2\in \B(L^2M)^{**}$ such that $[Q_1, p_{\rm nor}]=[Q_2, p_{\rm nor}]=0$, one has $Q_1\vee Q_2=Q_1+Q_2-Q_1Q_2$ commutes with $p_{\rm nor}$.

By Lemma~\ref{lem: identity} one has $q_\Y=\lim_F \iota_{\rm nor}(P_F)$, where $F\subset \Gamma$ is a finite set and $P_F\in \B(\ell^2\Gamma)$ is the orthogonal projection onto ${\rm sp}\{\delta_t\mid t\in F(\cup_{i=1}^n \Sigma_i)F\}$.
From the above discussion, we see that $P_F$ and $P$ are in the multiplicative domain of $\iota_{\rm nor}$.
Since $[P_F, P]=0$, we have $[\iota_{\rm nor}(P_F), \iota_{\rm nor}(P)]=0$
	and hence $[q_\Y, \iota_{\rm nor}(P)]=0$.
\end{proof}

Given a group $\Gamma$ with a family of subgroups $\cG$, recall that a set $S\subset \Gamma$ is small relative to $\cG$ if
	there exists some $n\in \N$, $\{s_i, t_i\}_{i=1}^n\subset \Gamma$ and $\{\Sigma_i\}_{i=1}^n\subset \cG$ such that
	$S\subset \bigcup_{i=1}^n s_i \Sigma_i t_i$, see e.g.\ \cite[Chapter 15]{BrOz08}.
We say a subgroup $\Lambda<\Gamma$ is almost malnormal relative to $\cG$ if $s\Lambda s^{-1}\cap \Lambda$ is small relative to $\cG$ for any $s\in \Gamma\setminus \Lambda$.
One checks that if $\cG$ only contains normal subgroups, then $\Lambda<\Gamma$ being almost malnormal relative to $\cG$ is equivalent to $s\Lambda t\cap \Lambda$ is small relative to $\cG$ for any $s,t\in \Gamma$ with at least one of $s, t$ is not in $\Lambda$.

\begin{lem}\label{lem: sum to identity}
Let $\Gamma$ be a countable discrete group with a subgroup $\Lambda<\Gamma$ and a family of subgroups $\{\Sigma_i\}_{i=1}^n$.
Set $M=L\Gamma$ and denote by $\X$ the $M$-boundary piece associated with $L\Lambda$,
	and by $\Y$ the $M$-boundary piece associated with $\{L\Sigma_i\}_{i=1}^n$.

Take $\{t_k\}_{k\in \N}\subset \Gamma$ a transversal for $\Gamma/\Lambda$
and put $p_{k,\ell}=P_{t_k \Lambda t_\ell}=\Ad(\lambda_{t_k} \rho_{t_\ell})(e_{L\Lambda})\in \B(L^2M)$,
where $P_{t_k \Lambda t_\ell}$ denotes the orthogonal projection from $\ell^2\Gamma$
	to $\overline{{\rm sp}\{\delta_g\mid g\in t_k \Lambda t_\ell\}}$.
	
Suppose $\Lambda$ is almost malnormal relative to $\{\Sigma_i\}_{i=1}^n$
	and each $\Sigma_i$ is normal.
Then $\{q_\Y^\perp \iota_{\rm nor}(p_{k,\ell})\}_{k,\ell\in \N}$ is a family of pairwise orthogonal projections
and $\sum_{k,\ell\in \N}q_\Y^\perp \iota_{\rm nor}(p_{k,\ell})=q_\Y^\perp q_\X$,
	where $q_\X$ and $q_\Y$ are identities of $(\K_\X(M)^\sharp_J)^*$ and $(\K_\Y(M)^\sharp_J)^*$, respectively.
\end{lem}
\begin{proof}
First note that $p_{k,\ell} p_{k',\ell'}=P_{t_k\Lambda t_\ell \cap t_{k'}\Lambda t_{\ell'}}\in \K_\Y(M)$ if $k\neq k'$ or $\ell\neq \ell'$,
	since $t_{k'}^{-1}t_k \Lambda t_\ell t_{\ell'}^{-1}\cap \Lambda$ is small relative to $\{\Sigma_i\}_{i=1}^n$.
As $\{p_{k,\ell}\}$ is in the multiplicative domain of $\iota_{\rm nor}$
and $[q_\Y, \iota_{\rm nor}(p_{k,\ell})]=0$ by Lemma~\ref{lem: commuting projections},
we have $q_\Y^\perp \iota_{\rm nor}(p_{k,\ell}) q_\Y^\perp \iota_{\rm nor}(p_{k',\ell'})=\delta_{k,k'}\delta_{\ell, \ell'} q_\Y^\perp \iota_{\rm nor}(p_{k,\ell})$.

To see $\sum_{k,\ell\in \N}q_\Y^\perp \iota_{\rm nor}(p_{k,\ell})=q_\Y^\perp q_\X$,
	we first notice that $\{p_n:=\vee_{k,\ell\leq n}p_{k,\ell}\}_{n\in \N}$
	is an approximate unit for the hereditary ${\rm C}^*$-subalgebra generated by 
	$\{xJyJe_{L\Lambda}\mid x,y\in C^*_r\Gamma\}$.
Thus by the same argument as in Lemma~\ref{lem: identity},
	we have $q_\X=\lim_n \iota_{\rm nor}(p_n)$
	and hence 
	$$q_\Y^\perp q_\X= \lim_n q_\Y^\perp \iota_{\rm nor}(\vee_{k,\ell\leq n} p_{k,\ell})=\sum_{k,\ell\in \N} q_\Y^\perp \iota_{\rm nor}(p_{k,\ell}).$$
\end{proof}

\subsection {Popa's intertwining-by-bimodules}

\begin{thm}[\cite{Po06B}] Let $(M,\tau)$ be a tracial von Neumann algebra and $P\subset pM p,Q\subset M$ be von Neumann subalgebras. 
Then the following  are equivalent:

\begin{enumerate}
\item There exist projections $p_0\in P, q_0\in Q$, a $*$-homomorphism $\theta:p_0P p_0\rightarrow q_0Q q_0$  and a non-zero partial isometry $v\in q_0M p_0$ such that $\theta(x)v=vx$, for all $x\in p_0P p_0$.

\item There is no sequence $u_n\in\mathcal U(P)$ satisfying $\|E_Q(x^*u_ny)\|_2\rightarrow 0$, for all $x,y\in pMp$.
\end{enumerate}

\end{thm}

If one of these equivalent conditions holds,  we write $P\prec_{M}Q$.

\subsection{Relative amenability}

Given a finite von Neumann algebra $M$ with a von Neumann subalgebra $Q\subset M$.
Recall from \cite{OzPo10I} that a von Neumann subalgebra $P\subset pMp$ with $p\in \cP(M)$ 
	is amenable relative to $Q$ in $M$ if there exists a u.c.p.\ map $\phi: p\langle M, e_Q\rangle p \to P$ 
		such that $\phi_{\mid pMp}=E_P$, where $E_P:pMp\to P$ is the normal conditional expectation.
Following \cite{IoPoVa13}, we say $P$ is strongly nonamenable relative to $Q$ in $M$ 
	if for any nonzero projection $p'\in P'\cap pMp$, we have $Pp'$ is not amenable relative to $Q$ in $M$.

The following is an abstraction of \cite[Corollary 2.12]{Ioa15} (also a straightforward relativization of \cite[Lemma 4.1]{Din22}),
	and thus we omit the proof.

\begin{lem}\label{lem: abstract Ioana}
Let $(M,\tau)$ and $N$ be tracial von Neumann algebras and $Q\subset M$ a von Neumann subalgebra.
% and $N\subset pMp$ von Neumann subalgebras with $p\in \cP(M)$.
Suppose there exist another tracial von Neumann algebra  $(\tilde M,\tilde \tau)$
	such that $M\subset \tilde M$ and $\tilde \tau_{\mid M}=\tau$,
	and a net of trace preserving automorphisms 
	$\{\alpha_t\}_{t\in \R}\subset \Aut(\tilde M)$
	such that  ${\alpha_t}_{\mid Q}\in  \Aut (Q)$,
	and such that ${\alpha_t}_{| M}\to \id_M$ in the point-$\|\cdot\|_2$ topology, as $t\to 0$.
Set $\tilde \alpha_t=\alpha_t\otimes\id_N\in \Aut(\tilde M\ovt N)$.

If a von Neumann subalgebra $P\subset p(M\ovt N)p$ is amenable relative to $Q\ovt N$ in $M\ovt N$, where $p\in \cP(M\ovt N)$ is some nonzero projection, 
	then for any $0<\delta \leq 1$, one of the following is true.
\begin{enumerate}
\item \label{item: uniform convergence}
There exists $t_\delta>0$ such that $\inf_{u\in \cU(P)}\|E_{M\ovt Q}(\tilde \alpha_{t_\delta}(u))\|_2>(1-\delta)\|p\|_2$.
\item \label{item: nonzero central vector in orthogonal}
There exists a net $\{\eta_k\}\subset \cK^\perp$, 
	where $\cK$ is the closure of $(M\ovt N) e_{Q\ovt N} (\tilde M\ovt N)$ inside $L^2(\langle \tilde M\ovt N, e_{Q\ovt N}\rangle)$, 
	such that $\|x\eta_k- \eta_k x\|_2\to 0$ for all $x\in P$,
	$\limsup_k \|y \eta_k\|_2\leq 2\|y\|_2$ for all $y\in p(M\ovt N)p$ and 
	$\limsup_k \|p \eta_k\|_2 >0$.
\end{enumerate}
\end{lem}

Specializing to the situation of Bernoulli actions, we obtain the following result.
%Here, we may use the deformation associated with Bernoulli actions from \cite{Po06B, Po06C} or \cite{Ioa}.

\begin{prop}\label{prop: intertwining from rel amen}
Let $\Gamma$ be a nonamenable group and denote by $M=L^\infty(X^\Gamma )\rtimes\Gamma$
	the von Neumann algebra associated with its Bernoulli action.
Let $N$ be a finite von Neumann algebra and $P\subset p( M\ovt N)p$ a von Neumann subalgebra with some nonzero projection $p\in M\ovt N$.
Suppose $P$ is amenable relative to $L\Gamma\ovt N$,
	and $P$ has no direct summand that is amenable relative to $1\otimes N$,
	then $P$ is rigid relative to deformation deformation associated with the Bernoulli action.

Moreover, if $\Gamma$ is i.c.c.\ and $N$ is a factor, then one of the following is true.
\begin{enumerate}
	\item There exists some partial isometry $v\in M\ovt N$ such that $v^* (\cN_{p(M\ovt N)p}(P)'' )v\subset L\Gamma\ovt N$ and $vv^*=p$.
	\item $\cN_{p(M\ovt N)p}(P)''\prec_{M\ovt N} L^\infty(X^\Gamma)\ovt N$.
\end{enumerate}
\end{prop}
\begin{proof}
Denote by $\tilde M$ and $\{\alpha_t\}\subset \Aut(\tilde M)$ from \cite{Po06B}.
It suffices to show that $(\ref{item: nonzero central vector in orthogonal})$ of Lemma~\ref{lem: abstract Ioana} does not occur by Popa's transversality lemma \cite[Lemma 2.1]{Po08}.
In fact, note that as $M\ovt N$ bimodules we have
$$L^2(\langle\tilde M\ovt N, {e_{L\Gamma \ovt N}}\rangle)\ominus \cK 
	=( L^2(\tilde M \ominus M) \otimes L^2N )\otimes_{L\Gamma \ovt N} (L^2\tilde M\otimes L^2 N)
	\prec  L^2M\otimes L^2N \otimes L^2 M,
$$
	since $L^2(\tilde M\ominus M)$ is a weakly coarse $M$-bimodule.
If $(\ref{item: nonzero central vector in orthogonal})$ were the case, then one would have almost $P$-central vectors in $L^2M\otimes L^2N\otimes L^2M$,
	which implies that $P$ has a direct summand is amenable relative to $1\otimes N$ in $M\ovt N$.

The moreover part follows from the same argument as in the beginning of the proof of \cite[Corollary 4.3]{IoPoVa13} by noticing 
	that $P\not\prec_{M\ovt N} 1\otimes N$ since $P$ has no direct summand amenable relative to $1\otimes N$ in $M\ovt N$ \cite[Remark 2.2]{Ioa15}.
%Set $Q=\cN_{p(M\ovt N)p}(P)''$ and assume $Q\not\prec_{M\ovt N} L^\infty(X^\Gamma)\ovt N$.
%Applying \cite[Theorem 2.7]{Bo13} to $Pr$ for any nonzero projection $r\in P'\cap (M\ovt N)$ yields $Pr\prec_{M\ovt N} L\Gamma \ovt N$ for all 
%	$r\in \cP(P'\cap (M\ovt N))$.
%The existence of some unitary $u\in \cU(M\ovt N)$ with $u^* Qu\subset L\Gamma\ovt N$ is then given by \cite[Proposition 2.3]{Bo13}.
\end{proof}

\subsection{Weak amenability and W*CMAP}
Recall from \cite{CH89} that a discrete group $\Gamma$ is weakly amenable if there exists a net of finitely supported functions $\varphi_i: \Gamma\to \C$
	such that $\varphi_i\to 1$ pointless,
	$m_{\varphi_i}(u_g)=\varphi_i(g)u_g$ extends to a c.b.\ map on $L\Gamma$
	and there exists some $C>0$ satisfying $\sup_i\|m_{\varphi_i}\|_{\rm cb}\leq C$.
The Cowling-Haagerup constant $\Lambda_{\rm cb}(\Gamma)$ is the infimum of all $C$ for which such a net $\varphi_i$ exists.
Similarly, a von Neumann algebra $M$ has the W$^*$ complete metric approximation property (W$^*$CMAP) if there exists a net of normal finite rank c.c.\ 
	$\theta_i: M\to M$ such that $\theta_i\to \id_M$ in the point-weak$^*$ topology.
Given a discrete group $\Gamma$,  one has $L\Gamma$ has W$^*$CMAP if and only if 
	$\Lambda_{\rm cb}(\Gamma)=1$.

\section{Biexactness and a Dichotomy of subalgebras}\label{sec: dichotomy}

In this section, we obtain a dichotomy of von Neumann subalgebras of $M\ovt N$, where $M$ is assumed to be biexact relative to some von Neumann subalgebra.
This can be seen as a relative version of \cite[Theorem 7.1]{DKEP22} and \cite[Proposition 2.3]{Din22}, where $N=\C$.
Since the arguments we employ here are somewhat ${\rm C}^*$-algebraic, we in addition assume $N$ has W$^*$CMAP.
This extra technical condition is used in the same way as in \cite[Proposition 7.3]{Iso20}, 
	and is similar to how exactness is used in \cite{Oz04}.
Specifically, we prove the following.

\begin{prop}\label{prop: dichotomy}
Let $M$, $N$ be finite von Neumann algebras and $Q\subset M$ a von Neumann subalgebra.
Suppose $M$ is biexact relative to $Q$ and $N$ has W$^*$CMAP.
Then for any von Neumann subalgebra $P\subset p( M\ovt N)p$ for some nonzero projection $p\in M\ovt N$, 
%	if $P$ is not properly proximal relative to $Q\ovt N$ in $M\ovt N$, then there exists a $P$-central functional $\varphi: \B(L^2M)\ovt N\to \C$ such that $\varphi_{\mid {M\ovt N}}$ is a normal state.
we have $P$ is either properly proximal relative to $Q\ovt N$ in $M\ovt N$, or there exists some nonzero projection $z\in \cZ(P)$ such that $Pz$ is amenable relative to $N$ in $M\ovt N$.
\end{prop}

We first consider the situation that $N$ is without W$^*$CMAP.

\begin{lem}\label{lem: dichotomy}
Let $M$, $N$ be finite von Neumann algebras and $Q\subset M$ a von Neumann subalgebra.
Suppose $M$ is biexact relative to $Q$.
Then for any von Neumann subalgebra $P\subset p( M\overline\otimes N)p$ with $0\neq p\in \cP(M\ovt N)$
	that is not properly proximal relative to $Q\overline\otimes N$ in $M\overline\otimes N$,
	there exist a nonzero central projection $z\in \cZ(P)$ 
	and a u.c.p.\ map $\phi\in \B(L^2M)\ovt N\to Pz$ such that $\phi_{\mid M\otimes_{\rm min} N}=\psi$, 
	where $\psi: M\ovt N\to Pz$ is a normal u.c.p.\ map with $\psi_{\mid Pz}=\id$.
%$E:M\ovt N\ni x\mapsto E_{Pz}(zxz)\in  Pz$ and $E_{Pz}: z(M\ovt N)z\to Pz$ is the normal conditional expectation.
\end{lem}
\begin{proof}
We first check that 
	$\bS_{\X_Q}(M)\otimes_{\rm min} N\subset \bS_{\X_{Q\overline\otimes N}}(M\overline\otimes N)$,
	for which it suffices to verify that $\K_{\X_Q}(M)^{\infty,1}\otimes_{\rm alg} N\subset \K_{\X_{Q\overline\otimes N}}^{\infty,1}(M\overline\otimes N)$.
Indeed, for any $a,b,c,d\in M$, $T\in \B(L^2M)$ and $x\in N$,
	note that 
	$$aJbJ e_Q T e_Q c JdJ\otimes x =(a\otimes 1)J(b\otimes 1)J e_{Q\ovt N}(T\otimes 1)e_{Q\ovt N} (c\otimes x) J(d\otimes 1)J\in \K_{\X_{Q\ovt N}}(M\ovt N).$$
Since $\B(L^2M)\ni T\mapsto T\otimes 1\in \B(L^2M\otimes L^2N)$ is continuous from both $M$-topology to $M\ovt N$-topology and $JMJ$-topology to $J(M\ovt N)J$-topology,
we have $\K_{\X_Q} ^{\infty,1}(M)\otimes_{\rm alg} N\subset 
	\K_{\X_{Q\overline\otimes N}} ^{\infty,1}(M\overline\otimes N)$.

Denote by $\theta_i: \B(L^2M)\to \bS_{\X_Q}(M)$ a net of normal u.c.p.\ maps such that ${\theta_i}_{\mid M}\to \id_M$ in the point $M$-topology, which is given by the assumption that $M$ is biexact relative to $Q$.
Consider 
	$$\tilde \theta_i:=\theta_i\otimes \id_N: \B(L^2M)\ovt N\to \bS_{\X_Q}(M)\otimes_{\rm min} N\subset \bS_{\X_{Q\ovt N}}(M\ovt N).$$

Since $P\subset p(M\ovt N)p$ is not properly proximal relative to $Q\ovt N$, 
	there exists a nonzero central projection $z\in P$ and a $P$-bimodular u.c.p.\ map $\psi: \bS_{\X_{Q\ovt N}}(M\ovt N)\to Pz$ 
	such that $\psi_{\mid M\ovt N}$ is normal.
Denote by $\phi$ a weak$^*$ limit point of 
	$\psi\circ \tilde \theta_i : \B(L^2M)\ovt N\to Pz$.

Now we show $\phi_{\mid M\otimes_{\rm min} N}=\psi$.
First note that for any $x\in M\otimes_{\rm min}N$, we have $\tilde \theta_i(x)\to x$ in the $M\ovt N$-topology.
To see this, one checks that $(\theta_i(a)-a)\otimes b\to 0$ in the the $M\ovt N$-topology for any $a\in M$ and $b\in N$.
Moreover, $\psi$ is continuous from the weak $M\ovt N$-topology to the weak operator topology on $Pz$ 
	as $\psi$ is u.c.p.\ and $\psi_{\mid M\ovt N}$ is normal.
Altogether, for any $x\in M\otimes_{\rm min} N$ and any $\omega\in (Pz)_*$, we have
$$\lim_i \langle \psi(\tilde \theta_i(x)), \omega\rangle=\lim_i \langle \tilde \theta_i(x), \omega\circ \psi\rangle=\langle x, \omega\circ \psi\rangle,$$
i.e., $\phi(x)=\lim_i\psi(\tilde\theta_i(x))=\psi(x)$.
\end{proof}

\begin{proof}[Proof of Proposition~\ref{prop: dichotomy}]
Since $N$ has W$^*$CMAP, there exists a net of finite rank normal c.c.\ maps 
	$\phi_n: N\to N$ such that $\phi_n\to \id_N$ in point-weak$^*$
	and 
$$\tilde \phi_n:=\id\otimes \phi_n: \B(L^2M)\ovt N\to \B(L^2M)\ot_{\rm min} N.$$
Consider $\phi\circ \tilde\phi_n: \B(L^2M)\ovt N\to Pz$, where $\phi:\B(L^2M)\ovt N\to Pz$ is from Lemma~\ref{lem: dichotomy},
denote by $\Phi$ a point weak$^*$ limit of $\phi\circ \tilde\phi_n$.
For any $x\in M\ovt N$ and $\varphi\in (Pz)_*$, since $\tilde\phi_n$ maps $M\ovt N$ to $M\otimes_{\rm min} N$ and
	$\tilde \phi_n\to \id_{M\ovt N}$ in the point weak$^*$ topology, we have
$$\langle \Phi(x),\varphi \rangle =\lim_n \langle \phi(\tilde\phi_n(x)),\varphi \rangle =\lim_n \langle  \tilde\phi_n(x), \varphi \circ \psi \rangle =\varphi(\psi(x)),$$
i.e., $\Phi_{\mid M\ovt N}=\psi$.
Therefore $\Phi: \B(L^2M)\ovt N\to Pz$ is a conditional expectation with $\Phi_{\mid M\ovt N}$ normal,
	which shows $Pz$ is amenable relative to $N$ in $M\ovt N$ by \cite{OzPo10I}.
\begin{comment}
Set $\psi:=\tau\circ \theta\in (\B(L^2M)\ovt N)^*$ and it is clear that $\psi$ is $P$-central.
We claim that $\psi$ is continuous in the $M\ovt N$-topology.
Indeed, for any $T\in\B(L^2M)\ovt N$,
$$M\ovt N\ni x\mapsto \psi(xT) $$

Replacing $\psi$ with its real or imaginary part, we may assume it is selfadjoint and
	take its Jordan decomposition $\psi=\psi_+-\psi_-$.
We claim that $\psi_+$ and $\psi_-$ 
	are both $P$-central and restrict to normal functionals on $M\ovt N$.
To see that $\psi_\pm$ is $P$-central, note that for any $u\in \cU(P)$,
	we have $\psi_+(\cdot)-\psi_-(\cdot)=\psi(\cdot)=\psi(u^*\cdot u)=\psi_+(u^*\cdot u)-\psi_-(u^*\cdot u)$.
	Since $\|\psi\|\leq \|\psi_+(u^*\cdot u)\|+\|\psi_-(u^*\cdot u)\|\leq \|\psi_+\|+\|\psi_-\|=\|\psi\|$,
	by the uniqueness of Jordan decomposition, we have $\psi_\pm(u^* \cdot u)=\psi_+(\cdot)$
	and hence $\psi_\pm$ is $P$-central.

Similarly, take a net of contractions $\{x_i\}\subset (M\ovt N)_+$ that converges to $1$ strongly.
For any $T\in \B(L^2M)\ovt N$, we have $\psi(x_i^{1/2} T x_i^{1/2})\to \psi(T)$.
Set $\varphi^\pm:=\lim_i\psi_\pm(x_i^{1/2}\cdot x_i^{1/2})$ and we have $\psi=\varphi^+-\varphi^-$ and $\|\varphi^\pm\|\leq \|\psi^\pm\|$.
It follows that $\varphi^\pm=\psi_\pm$ and hence $\psi_\pm(x_i)\to \psi_{\pm}(1)$.
Since $\{x_i\}$ is arbitrary, we have $\psi_\pm$ are normal when restricted to $M\ovt N$.

Thus we conclude $Pp$ is amenable relative to $Q\ovt N$ in $M\ovt N$ for some nonzero projection
	$p\in \mathcal Z(P'\cap (M\ovt N))$.
\end{comment}
\end{proof}

Let $M$ be a finite von Neumann algebra and $N\subset M$ a von Neumann subalgebra.
Recall that $M$ is solid relative to $N$ if for any von Neumann subalgebra $P\subset M$ such that $P\not\prec_M N$, one has $P'\cap M$ is amenable relative to $N$ in $M$.
%We say $M$ is strongly solid relative to $N$ if for any von Neumann subalgebra $P\subset M$ with $P$ amenable relative to $N$ in $M$, one either has $\cN_M(P)$ is also amenable relative to $N$ or $P\prec_M N$.
The following recovers \cite[Proposition 7.3]{Iso20}.

\begin{cor}
Let $M$ and $N$ be finite von Neumann algebras.
If $M$ is biexact relative to a von Neumann subalgebra $Q\subset M$ 
	and $N$ has W$^*$CMAP, then $M\overline\otimes N$ is solid relative to $Q\ovt N$.
\end{cor}
\begin{proof}
Suppose $P\subset M\ovt N$ is a von Neumann subalgebra such that $P\not\prec_{M\ovt N} Q\ovt N$, from which we obtain a sequence of unitary $\{u_n\}\subset \cU(P)$ with $\|E_{Q\ovt N}(x u_n y)\|_2\to 0$ for any $x,y\in M\ovt N$.
As in \cite{DKEP22}, a point weak$^*$ limit point $\phi:=\lim_n \Ad(u_n): \bS_{\X_{Q\ovt N}}(M\ovt N)\to P'\cap (M\ovt N)$
	gives a conditional expectation.
Then the proof of Proposition~\ref{prop: dichotomy} shows that  $P'\cap (M\ovt N)$ is amenable relative to $N$ in $M\ovt N$ and 
hence amenable relative to $Q\ovt N$ in $M\ovt N$.
\end{proof}
%
%\begin{lem}
%Let $M$ be a finite von Neumann algebra with $N\subset M$ a von Neumann subalgebra. 
%Suppose that for any von Neumann subalgebra $P\subset M$, we have $P$ is either amenable relative to $N$ in $M$ or properly proximal relative to $N$ in $M$.
%Then $M$ is solid relative to $N$.
%\end{lem}
%\begin{proof}
%Let $P\subset M$ be a von Neumann subalgebra with $P\not \prec_{M} N$.
%Then $P'\cap M$ is not properly proximal relative to $N$ in $M$ by considering $\varphi:=\lim_n\langle \cdot \hat u_n, \hat u_n\rangle\in \B(L^2(M\ovt N))^*$,
%	where $\{u_n\}\subset \cU(P)$ is a sequence given by $P\not\prec_{M\ovt N} N$.
%Therefore $P'\cap M$ is amenable relative to $N$ in $M\ovt N$.
%\end{proof}

\section{Concentration of states on the boundary and relative amenability}
\label{sec: concentration of state}

In this section, we show a relative version of the main technical result in \cite{Din22} and \cite{DKE22}.
The following lemma shows that under certain conditions, one may relate the basic construction and the small-at-infinity boundary.

\begin{lem}\label{lem: embed basic construction}
Let $\Gamma$ be a countable discrete group with subgroups $\Lambda$ and $\Sigma_i$, $1\leq i\leq n$.
Suppose $\Lambda<\Gamma$ is almost malnormal relative to the family $\{\Sigma_i\}_{i=1}^n$
	and each $\Sigma_i$ is normal.
Denote by $\X$ the $L\Gamma$-boundary piece associated with $L\Lambda$ and 
	$\Y$ the $L\Gamma$-boundary piece associated with $\{L\Sigma_i\}_{i=1}^n$.
Then there exists a u.c.p.\ map $\phi:\langle L\Gamma, e_{L\Lambda}\rangle \to q_\Y^\perp \tilde \bS_\Y(L\Gamma) q_\Y^\perp $
such that $\phi(x)=q_\Y^\perp q_\X \iota_{\rm nor}(x)$ for any $x\in L\Gamma$,
where $q_\X$ and $q_\Y$ denote the identity of $\K_\X(L\Gamma)^{\sharp *}_J$ and
	$\K_\Y(L\Gamma)^{\sharp *}_J$, respectively.
\end{lem}

\begin{proof}
We follow the proof in \cite{Din22} closely.
Denote by $\{t_k\}_{k\geq 1} \subset \Gamma$ a transversal of $\Gamma/\Lambda$.
Set $M=L\Gamma$, $N=L\Lambda$ and $u_k=\lambda_{t_k}\in \cU(M)$.

For each $n\geq 1$, consider the c.p.\ map $\psi_n: \langle M, e_N\rangle \to \langle M, e_N\rangle$
	given by 
	$$\psi_n(x)=(\sum_{k\leq n} u_k e_N u_k^*)x(\sum_{j\leq n} u_j e_N u_j^*).$$
Notice that the image of $\psi_n$ lies in the $*$-algebra $A_0:={\rm sp}\{u_k a e_N u_j^*\mid a\in N,\ j,k\geq 1\}$.

By \cite[Lemma 3.4]{Din22} , we have $\{\iota_{\rm nor} (Ju_kJ e_N Ju_k^* J)\}_{k\geq 1}\subset \B(L^2M)^{\sharp *}_J$
	is a family of pairwise orthogonal projections 
	and let $e=\sum_{k\geq 1} \iota_{\rm nor} (Ju_kJ e_N Ju_k^* J)$.
Set $\pi: A_0\to q_\Y^\perp \B(L^2M) ^{\sharp *}_J q_\Y^\perp$ to be the linear map satisfying
	$\pi(u_k a e_N u_j^*)=q_\Y^\perp \iota_{\rm nor}(u_ka) e\iota_{\rm nor}(u_j^*) q_\Y^\perp$.

We first show that $\pi$ is a $*$-homomorphism. 
Indeed, we claim that for any $x\in M$, we have
\begin{equation}\label{equa: homomorphism}
q_\Y^\perp e \iota_{\rm nor}(x) e q_\Y^\perp =q_\Y^\perp \iota_{\rm nor}(E_N(x))e q_\Y^\perp.
\end{equation}
By normality of $\iota_{\rm nor}$ and $E_N$, it suffices to check this for $x\in \C\Gamma$.
Compute 
\[
\begin{aligned}
q_\Y^\perp e \iota_{\rm nor}(x) e q_\Y^\perp
	& = \Ad(q_\Y^\perp)( \sum_{k,j}\iota_{\rm nor}(  (Ju_k J e_N Ju_k^* J) x (Ju_jJ e_N Ju_j^*J))\\
	& =\Ad (q_\Y^\perp)(\iota_{\rm nor}(E_N(x))e + 
	\Ad (q_\Y^\perp)(\sum_{k\neq j} \iota_{\rm nor}((Ju_kJ e_N Ju_k^*J) x(Ju_jJ e_N Ju_j^*J))).
\end{aligned}
\]
Note that for $x=\lambda_t\in \C\Gamma$, one has 
	$(Ju_k J e_N Ju_k^* J) x (Ju_jJ e_N Ju_j^*J) = \lambda_t P_{t^{-1}\Lambda t_k\cap \Lambda t_j}$
	which is in $\K_\Y(M)$ if $k\neq j$ since $\Lambda$ is almost malnormal relative to $\{\Sigma_i\}_{i=1}^n$
	and hence the (\ref{equa: homomorphism}) follows.
Here, for a subset $S\subset \Gamma$, $P_S$ denotes the orthogonal projection from $\ell^2\Gamma$ to $\overline{{\rm sp}}\{\delta_g\in g\in S\}$.

Moreover, we have that $[e, q_\Y]=0$.
To see this, recall from Lemma~\ref{lem: identity} that $q_\Y=\lim_F \iota_{\rm nor}(e_F)$, where $\{e_F\}\in \K_\Y(M)$ is an approximate unit of the form that $e_F$ is the orthogonal projection on to ${\rm sp}\{\delta_g\mid g\in F(\cup_{i=1}^n \Sigma_i) F\}$ for any finite subset $F\subset \Gamma$.
Since $Ju_k Je_N Ju_k^*J$ is the projection on to ${\rm sp}\{\delta_g\mid g\in \Lambda t_k\}$,
	we have $[\iota_{\rm nor}(Ju_k Je_N Ju_k^*J), q_\Y]=0$ by Lemma~\ref{lem: commuting projections}
	and hence $[e, q_\Y]=0$ as well.
Together with (\ref{equa: homomorphism}), it then follows that $\pi$ is a homomorphism.

We then show $\pi$ is a contraction and thus $\pi$ extends to the ${\rm C}^*$-algebra
	 $A:=\overline {A_0}^{\|\cdot\|}$.
To this end, it suffices to show that for any $\sum_{i=1}^d u_{k_i} a_i e_N u_{j_i}^*\in A_0$
	with $a_i\in N$ and $k_i, j_i\geq 1$, and unit vectors $\xi,\eta\in \cH$
	we have $|\langle \pi(\sum_{i=1}^d u_{k_i} a_i e_N u_{j_i}^*) \xi, \eta\rangle | \leq \|\sum_{i=1}^d u_{k_i} a_i e_N u_{j_i}^*\|$,
	where $\cH$ is the Hilbert space on which $\B(L^2M) ^{\sharp *}_J$ is represented.

For each $k\in \N$, set $Q_k= q_\Y^\perp \sum_{i=1}^d \iota_{\rm nor}(P_{t_{j_i}\Lambda t_k^{-1}})$ and $R_k=q_\Y^\perp \sum_{i=1}^d \iota_{\rm nor}(P_{t_{k_i}\Lambda t_k^{-1}})$.
Note that $\{Q_k\}$ and $\{R_k\}$ are two families of pairwise orthogonal projections as for $k\neq r$ we have $Q_k Q_r=q_\Y^\perp \sum_{i, j=1}^d \iota_{\rm nor}(P_{t_{j_i}\Lambda t_k^{-1}\cap t_{j_j}\Lambda t_r^{-1}})=0$
	and similarly $R_k R_r=0$.
Moreover, observe that 
\[\begin{aligned}
\sum_{i=1}^d \iota_{\rm nor}(e_N u_{j_i}^* Ju_k^*J) Q_k= q_\Y^\perp \sum_{i=1}^d \iota_{\rm nor}(e_N u_{j_i}^* Ju_kJ),\ 
\sum_{i=1}^d \iota_{\rm nor} (e_N u_{k_i}^* Ju_k^* J) R_k=q_\Y^\perp \sum_{i=1}^d \iota_{\rm nor} (e_N u_{k_i}^* Ju_k^* J).
\end{aligned}\]
Thus we compute
\[\begin{aligned}
|\langle \pi(\sum_{i=1}^d u_{k_i} a_i e_N u_{j_i}^*) \xi, \eta\rangle | 
	&\leq \sum_{k\geq 0} |\sum_{i=1}^d \langle q_\Y^\perp \iota_{\rm nor}(e_N u_{j_i}^* Ju_k^*J) \xi, q_\Y^\perp \iota_{\rm nor} (Ju_kJ u_{k_i} e_Na_i)^*\eta\rangle|\\
	&=  \sum_{k\geq 0} |\sum_{i=1}^d \langle q_\Y^\perp \iota_{\rm nor}(e_N u_{j_i}^* Ju_k^*J) Q_k \xi, q_\Y^\perp \iota_{\rm nor} (Ju_kJ u_{k_i} e_Na_i)^*R_k\eta\rangle|\\
	& \leq \sum_{k\geq 0} \|\iota_{\rm nor}( Ju_k J (\sum_{i=1}^d u_{k_i}a_i e_N u_{j_i}^*) Ju_k^*J)\|\|Q_k\xi\|\|R_k \eta\|\\
	&\leq \|\sum_{i=1}^d u_{k_i} a_i e_N u_{j_i}^*\|.
\end{aligned}\]

Next we show $[\pi(A_0),\iota_{\rm nor}(x)]=0$ for any $x\in JMJ$, and thus 
$$\psi_n\circ \pi: \langle M, e_N\rangle \to q_\Y^\perp \tilde\bS_\Y(M) q_\Y^\perp.$$
%We have $[\phi(\langle M, e_N\rangle), \iota_{\rm nor}(x)]=0$ for any $x\in JMJ$.
For this, it suffices to check that $\iota_{\rm nor}(Ju_s^*J) e \iota_{\rm nor}(Ju_sJ)=e$ for any $s\in \Gamma$.
Indeed, observe that $e=\sum_{k\geq 1} \iota_{\rm nor} (Ju_kJ e_N Ju_k^* J)$
	is independent of the choice of traversals of $\Gamma/\Lambda$
	and $\Gamma=\sqcup_{k\in \N} t_k \Lambda=\sqcup_{k\in \N} s t_k \Lambda$.

Finally, consider the sub-unital c.p.\ map
	 $\phi_n:=\psi_n\circ \pi: \langle M, e_N\rangle \to  q_\Y^\perp \tilde\bS_\Y(M) q_\Y^\perp$,
	and denote by $\phi$ a weak$^*$ limit point of $\phi_n$.
We claim that $\phi(x)=q_\Y^\perp q_\X \iota_{\rm nor}(x)$ for any $x\in M$.

In fact, for any $x\in M$, we have 
\[
\begin{aligned}
    \phi(x)=&\lim_{n\to\infty} \pi\Big(\sum_{0\leq k,\ell\leq n} (u_k  E_{N}( u_k^* x u_\ell)e_{N}  u_\ell ^*)\Big)
    = q_{\Y}^{\perp} \lim_{n\to\infty} \sum_{0\leq k,\ell\leq n} \iota_{\rm nor}( u_k  E_{N}( u_k^* x u_\ell) ) e\iota_{\rm nor}( u_\ell ^*)\\
	=&  q_{\Y}^{\perp} \lim_{n\to\infty} \sum_{0\leq k,\ell\leq n} \big ( \iota_{\rm nor} (u_k) e \iota_{\rm nor}(u_k^*)\big ) 
	\iota_{\rm nor}( x)\big( \iota_{\rm nor}( u_\ell )e \iota_{\rm nor}(u_\ell^*)\big),\\
\end{aligned}
\]
where the last equation follows from (\ref{equa: homomorphism}).
Finally, note that by Lemma~\ref{lem: sum to identity}, $\{p_k\}_{k\geq 0}$ is a family of pairwise orthogonal projections, where
$$p_k:=q_\Y^\perp  \iota_{\rm nor} (u_k) e \iota_{\rm nor}(u_k^*)=q_\Y^\perp \sum_{r\geq 0} \iota_{\rm nor}(Ju_r J u_k e_N u_k^* Ju_r^*J),$$
	and $\sum_{k\geq 0}p_k= \sum_{k,r\geq 0}q_\Y^\perp \iota_{\rm nor}(Ju_r J u_k e_B u_k^* Ju_r^*J)=q_\Y^\perp q_\X$.
Therefore, we conclude that $\phi(x)=q_\Y^\perp q_\X \iota_{\rm nor}(x)$, as desired.
\end{proof}

\begin{prop}\label{prop: non prop prox to rel amen}
Let $\Gamma$ be a countable discrete group with two families of subgroups $\{\Lambda_i\}_{i=1}^n$ and $\{\Sigma_i\}_{i=1}^n$ such that each $\Lambda_i$ is almost malnormal relative to $\{\Sigma_j\}_{j=1}^n$
	and and each $\Sigma_i$ is normal.
Set $M=L\Gamma$ and denote by $\X$ and $\Y$ the $M$-boundary pieces associated with $\{L\Lambda_i\}_{i=1}^n$ and $\{L\Sigma_i\}_{i=1}^n$, respectively.

Suppose $N\subset p M p$, for some nonzero $p\in \cP(M)$, is a von Neumann subalgebra that has no amenable direct summand,
	such that $N$ is properly proximal relative to $\{L\Lambda_i\}_{i=1}^n$ in $M$.

If there exists an $N$-central state $\varphi: \tilde \bS_\Y(M)\to \C$ such that $\varphi_{\mid pMp}$ is a faithful normal state,
%
%	and $N$ is not properly proximal relative to $\{L\Sigma_i\}_{i=1}^n$ in $M$.
%If each $\Lambda_i$ is almost malnormal relative to $\{\Sigma_j\}_{j=1}^n$, 
	then there exists a partition of unity $\{p_i\}_{i=1}^n\subset \cP(\cZ(N'\cap p Mp))$
	such that $Np_i$ is amenable relative to $L\Lambda_i$ in $M$ for each $1\leq i\leq n$.
\end{prop}
\begin{proof}
%Denote by $\X$ and $\Y$ the boundary pieces associated with $\{L\Lambda_i\}_{i=1}^n$ and $\{L\Sigma_i\}_{i=1}^n$, respectively.
%Since $N\subset M$ is not properly proximal relative to $\{L\Sigma_i\}_{i=1}^n$,
%	there exists an $N$-central state $\varphi$ on 
%	$$ \tilde \bS_{\Y}(M):=\{T\in (\B(L^2M)^\sharp_J)^*\mid
%	 [T, JxJ]\in (\K_{\Y}(M)^\sharp_J)^*{\rm \ for\ any\ } x\in M\}.$$

%Recall that $(\K_{\Y}(M)^\sharp_J)^*$ is a nonunital von Neumann subalgebra of $(\B(L^2M)^\sharp_J)^*$,
%	and we may set $q_{\Y}$ to be the identity of $(\K_{\Y}(M)^\sharp_J)^*\subset (\B(L^2M)^\sharp_J)^*$.

Denote by $q_\X$ and $q_\Y$ the identities of $\K_\X(M) ^{\sharp *} _J$ and $\K_\Y(M) ^{\sharp *} _J$, respectively.

Note that $\varphi(q_{\Y})=0$.
Indeed, if this is not the case, we may then consider the $M$-bimodular u.c.p.\ map
	$$\B(L^2M)\ni T \mapsto q_\Y \iota_{\rm nor}(T)q_\Y\in q_\Y \B(L^2M)^{\sharp *}_J q_\Y.$$
Since $\K_{\Y}(M)\subset \B(L^2M)$ is a hereditary ${\rm C}^*$-algebra with $M$ contained in its multiplier algebra,
	we see that $q_{\Y} \B(L^2M)^{\sharp *}_J q_{\Y}= \K_{\Y}(M)^{\sharp *}_J$ and $[q_{\Y}, M]=0$.
The non-vanishing of $\varphi$ on $q_{\Y}$ then yields an $N$-central state on $\B(L^2M)$ 
	which entails that $N$ has an amenable direct summand.
%In particular, $\varphi$ restricts to a state on $q_{\Y}^\perp \tilde \bS_{\Y}(M) q_{\Y}^\perp$.

On the other hand, notice that $q_\X\in \tilde \bS_\Y(M)$ as $[q_\X, JMJ]=0$, 
	and we have $\varphi(q_\X^\perp)=0$ since $N$ is properly proximal relative to $\X$.
In fact, it is clear that
 $$q_\X^\perp \tilde \bS_\X(M) q_\X^\perp \subset \B(L^2M)^{\sharp *}_J\cap (JMJ)'
	\subset \tilde \bS_{\Y}(M)$$
	and hence restricting $\varphi$ to $q_\X^\perp \tilde \bS_\X(M) q_\X^\perp$ gives us an $N$-central state on $\tilde \bS_\X(M)$.
As $\iota_{\rm nor}(p\bS_\X(M)p)\subset \tilde \bS_\X(M)$, we yields an $N$-central state on $p\bS_\X(M)p$ that restricts to a normal state on $pMp$, contradicting the assumption
	that $N$ is properly proximal relative to $\X$ in $M$.

From the above discussion, we conclude that $\varphi(q_{\Y}^\perp q_\X )=1$.
Furthermore, if we denote by $\X_i$ the boundary piece associated with $L\Lambda_i$ and 
	by $q_i$ the identity of $\K_{\X_i}(M)^{\sharp *}_J$,
	we then have $\varphi(q_\Y^\perp q_i q_\Y^\perp)> 0$ for some $1\leq i\leq n$.
This is a consequence of the facts that $q_\X=\vee_{i=1}^n q_i$ and $\{q_i\}_{i=1}^n$ are pairwise commuting
	by \cite[Lemma 3.10]{DKE22}.

By Lemma~\ref{lem: embed basic construction}, we have an $M$-bimodular map $\phi: \langle M, e_{L\Lambda_i}\rangle \to 
	q_\Y^\perp \tilde\bS_\Y(M) q_\Y^\perp $ for each $i$.
Composing with $\psi_i:=\varphi( \cdot )/\varphi(q_\Y^\perp q_i q_\Y ^\perp)$ yields an $N$-central state which is normal on $M$
	and hence we may find some projection $p_i\in \cZ(N'\cap pMp)$
	such that $N p_i$ is amenable relative to $L\Lambda_i$ in $M$,
	where $p_i$ is the support of $\psi_i$ on $\cZ(N'\cap pMp)$.

Lastly, we claim that $\vee p_i=p$.
Indeed, observe that $\varphi(q_ip_i^\perp )=0$ and $[p_i q_i, p_j q_j]=0$
	and hence $\varphi(\vee p_i)\geq \varphi(\vee p_i q_i)=\varphi(\vee q_i)=\varphi(p)$.
Since $\varphi_{\mid pMp}$ is a faithful normal state, it follows that $\vee p_i=p$.
\end{proof}

\section{Some properties of the comultiplication map}\label{sec: comultiplication}
The comultiplication map $\Delta$ was first considered in \cite{PoVa10a} and has been an indispensable tool to obtain superrigidity results (e.g. \cite{Ioa11,IoPoVa13}).
Since our proof of Theorem~\ref{thm: recover Bernoulli} follows the strategy of \cite{IoPoVa13} closely, 
	we collect some properties of the comultiplication map in relation to proper proximality in this section.

Recall that if a von Neumann algebra $M$ is isomorphic to $L\Lambda$ for some group $\Lambda$,
	then the associated comultiplication is given by
	$$\Delta: M\ni  u_t\to  u_t\otimes u_t\in M\ovt M,$$
	for $t\in \Lambda$,
	where $u_t\in L\Lambda$ is the canonical unitary corresponding to $t\in \Lambda$.

Observe that the comultiplication extends to an isometry  
	$V:L^2M\ni \delta_t\mapsto \delta_t\otimes \delta_t \in L^2M\otimes L^2M$,
	and we denote by $\phi=\Ad(V): \B(L^2M\ovt L^2M)\to \B(L^2M)$ the corresponding u.c.p.\ map.
%Note that $\phi\circ \Delta(x)=x$ for any $x\in M$.

\begin{lem}\label{lem: comultiplication}
Let $M\cong  L\Lambda$, $\Delta$ and $\phi$ be as above.
Then the following are true.
\begin{enumerate}
\item We have $\phi(x)=(\Delta^{-1} \circ  E_{\Delta(M)})(x)$
	and $\phi(JxJ)=J(\Delta^{-1} \circ  E_{\Delta(M)})(x)J$
	for any $x\in M\ovt M$,
	where $E_{\Delta(M)}: M\ovt M\to \Delta(M)$ is the normal conditional expectation.
	\label{item: image}
%We have $\phi(M\ovt M)\subset M$, $\phi(J(M\ovt M)J)\subset JMJ$, $\phi(\Delta(x))=x$ and $\phi(J\Delta(x)J)=JxJ$ for any $x\in M$.; \label{item: image}
\item Both $\phi(\K(L^2M)\mot \B(L^2M))$ and $\phi(\B(L^2M)\mot\K(L^2M))$ are in $\K(L^2M)$. \label{item: compact image}
\item The map $\phi$ is continuous in $\|\cdot\|_{\infty,1}$ and $\|\cdot\|_{\infty,2}$. \label{item: infty-1 continuity}
\end{enumerate}
\end{lem}

\begin{proof}
(\ref{item: image}) is straightforward to check.
%It is easy to check that $V^* (u_t\otimes u_s) V = \delta_{t,s} u_t$, 
%	$V^* (Ju_tJ\otimes Ju_sJ) V=\delta_{t,s} Ju_tJ$.
%Then the desired result follows from the normality of $\phi$.

(\ref{item: compact image})
Let $\{\xi_n\}\subset (L^2M)_1$ be an sequence that converges to $0$ weakly
	and write each $\xi_n=\sum_{t\in \Lambda} \alpha_{n,t} \delta_t$.
Since $\xi_n\to 0$ weakly, for any finite subset $F\subset \Lambda$,
	we have $\lim_{n\to \infty}\|P_F\xi_n\|=\lim_{n\to \infty} (\sum_{t\in F}|\alpha_{n,t}|^2)^{1/2}=0$,
	where $P_F\in \K(\ell^2\Lambda)$ is the finite rank projection corresponding to $F$.
% and $\varepsilon>0$,
%	there exists some $n_0\geq 0$ such that 
%	$\sum_{t\in F} |\alpha_{n,t}|^2<\varepsilon$ for any $n\geq n_0$.

%For a finite subset $F\subset \Lambda$, take $P_F\in \K(\ell^2\Lambda)$ to be the finite rank projection corresponding to $F$.
Then for any contraction $T\in \B(L^2M)$, we have 
$$ \| (P_F\otimes T)V\xi_n\|=\|(1\otimes T)(\sum_{t\in F} \alpha_{n,t} \delta_t\otimes \delta_t) \| \leq (\sum_{t\in F}|\alpha_{n,t}|^2)^{1/2}\to 0, $$
as $n\to \infty$.
It then follows that $\phi$ maps $\K(L^2M)\mot \B(L^2M)$ and $\B(L^2M)\mot\K(L^2M)$ to $\K(L^2M)$.

(\ref{item: infty-1 continuity})
For any $T\in \B(L^2M\ovt M)$, note that
$$ \|\phi(T)\|_{\infty, 1}=\sup_{a,b\in (M)_1} \langle TV\hat a, V\hat b\rangle =\sup_{a,b\in (M)_1}\langle T\widehat{\Delta(a)} , \widehat{\Delta(b)}\rangle\leq \|T\|_{\infty,1}.$$
The proof for $\|\cdot\|_{\infty,2}$ is similar.
\end{proof}

\begin{comment}

\begin{lem}
Let $M$ be a finite von Neumann algebra with von Neumann subalgebra $P$ and $Q$.
Suppose $Q$ is regular and $P\not\prec_M Q$.
If $P$ is not properly proximal, then $P$ is not properly proximal relative to $Q$ in $M$.
\end{lem}

\begin{proof}
Pick a sequence $u_n\in \cU(P)$ given by $P\not\prec_M Q$ and consider 
	$\theta_n=\Ad(u_n)\in UCP( (\B(L^2M)^\sharp_J)^*)$.
For any $K\in \K_{\X_Q}(M)$, we have $\|u_n^* K u_n\|_{\infty,1}\to 0$ as $n\to \infty$
	and hence $\theta_n(K)\to 0$ in the weak$^*$ topology in $(\B(L^2M)^\sharp_J)^*$.
Denoting by $\theta$ a weak$^*$ limit point of $\theta_n$, we then have $\theta$ vanishes on 
	$(\K_{\X_Q}(M)^\sharp_J)^*$.

For any $T\in \tilde \bS_{\X_Q}(M)$, $\varphi\in \B(L^2M)^\sharp_J$ and $x\in M$, we check that 
$$\varphi([\theta(T), JxJ])=\lim_n \varphi([\theta_n(T), JxJ])=\lim_n\varphi(\theta_n([T, JxJ]))=0,$$
and hence $\theta$ maps $\tilde \bS_{\X_Q}(M)$ to $\tilde \bS(M)$.

Since $P$ is not properly proximal, there exists a $P$-central state $\psi$ on $\tilde \bS(M)$
	that is normal when restricted to $M$.
We claim that $\tilde \psi:=\psi\circ \theta$ is also $P$-central and $\tilde \psi_{\mid M}$ is normal.
For any $u\in \cU(P)$, denote by $\theta_u$ the weak$^*$ limit of $\Ad(u^* u_n u)$.
For $T\in \tilde \bS_{X_\Q}(M)$, 
	notice that $\theta(u^* Tu)= u^*\theta_u(T)u$
\end{proof}

\textcolor{blue}
{If $e_P e_Q e_P$ is compact, then this is clear.}

\textcolor{blue}
{Maybe using Marrakchi's characterization of binormal states, i.e.,
	working with $\B(L^2(M^\cU))$ would be easier?}
\end{comment}

\begin{lem}\label{lem: compact comultiplication}
Let $\Lambda$ be a countable discrete group and $M=L\Lambda$.
Denote by $\Delta: L\Lambda\to L\Lambda\ovt L\Lambda$ the comultiplication map
	and $\X$ the $M\ovt M$-boundary piece associated with $\{M\otimes 1, 1\otimes M\}$.
Then we have $\phi: \K_\X^{\infty,1}(M\ovt M)\to \K^{\infty,1}(M)$, where $\phi$ is the u.c.p.\ map defined above.
%
%
%Let $M$ be a finite von Neumann algebra such that $M\cong L\Lambda$ for some group $\Lambda$
%	and $N\subset M$ a von Neumann subalgebra.
%Suppose $N$ is not properly proximal in $M$.
%Then $\Delta(N)$ is not properly proximal relative to $\{M\otimes 1, 1\otimes M\}$ in $M\ovt M$.
\end{lem}
\begin{proof}
%Let $\phi:\B(L^2M\otimes L^2M)\to \B(L^2M)$ be as above.
%If we denote by $\X$ the $M\ovt M$-boundary piece associated with $\{M\otimes 1, 1\otimes M\}$,
%	it then suffices to show that $\phi$ restricts to a map from $\bS_\X(M\ovt M)$ to $\bS(M)$.
%Indeed, since $N$ is not properly proximal in $M$, there exists an $N$-central state $\varphi$ on $\bS(M)$ such that $\varphi_{\mid M}$ is normal.
%We may then consider $\varphi\circ \phi$, which yields a $\Delta(N)$-central state on $\bS_\X(M\ovt M)$ with $\varphi\circ\phi$ restricts to a normal state on $M\ovt M$
%	by Lemma~\ref{lem: comultiplication}, (\ref{item: image}).
%
Set $\tilde \Lambda=(\Lambda\times\{e\}) \cup (\{e\}\times \Lambda)\subset \Lambda\times\Lambda$
and denote by $P_S:\ell^2(\Lambda\times \Lambda)\to \overline{{\rm sp}\{\delta_t\mid t\in S\}}$
	the orthogonal projection for any subset $S\subset \Lambda\times \Lambda$.

By Lemma~\ref{lem: comultiplication}, (\ref{item: compact image}),
	we have $\phi(P_{F\tilde \Lambda F})\in \K(L^2M)$ for any finite subset $F\subset \Lambda$.
Denote by $\X_0$ the hereditary ${\rm C}^*$-subalgebra generated by 
	$\{xJyJ P_{\tilde \Lambda}\mid x,y\in C^*_r(\Lambda\times \Lambda)\}$.
By the proof of \cite[Lemma 3.5]{Din22},	
	we have $\X_0\subset \K_\X^{\infty,1}(M\ovt M)$ is dense in $\|\cdot\|_{\infty,1}$.
	
One checks that $\{P_{F\tilde \Lambda F}\}_F$ forms an approximate unit in $\X_0$
	and hence $\phi(\X_0)\subset \K(L^2M)$.
Moreover, since $\phi$ is continuous in $\|\cdot\|_{\infty,1}$ and 
	$\X_0\subset \K_\X^{\infty,1}(M\ovt M)$ is dense in $\|\cdot\|_{\infty,1}$ by Lemma~\ref{lem: comultiplication}, (\ref{item: infty-1 continuity}),
	we conclude $\phi( \K_\X^{\infty,1}(M\ovt M))\subset \K^{\infty,1}(M)$.
%	
%It then follows from the facts that $\{P_{F\tilde \Lambda F}\}$ forms an approximate unit in $\K_\X^{\infty,1}(M\ovt M)$ in the $\|\cdot\|_{\infty,1}$
%	and (\ref{item: infty-1 continuity}) of Lemma~\ref{lem: comultiplication} we have  $\phi(\K_\X^{\infty,1}(M\ovt M))\subset \K^{\infty,1}(M)$.
%
%Therefore, if $T\in \bS_\X(M\ovt M)$, we then have $[\phi(T), JxJ]=\phi([T, J\Delta(x) J])\in \K^{\infty,1}(M)$ for any $x\in M$, i.e., $\phi:\bS_\X(M\ovt M)\to \bS(M)$ is a u.c.p.\ map such that $\phi\circ \Delta_{\mid M}=\id_M$.
\end{proof}

Given a finite von Neumann algebra $M$ with a von Neumann subalgebra $N$, we denote by $E_N^M: M\to N$ the normal conditional expectation.

\begin{cor}\label{cor: non prop prox comultiplication 1}
Let $\Lambda$ be a countable discrete group and $N\subset L\Lambda=:M$
	a von Neumann subalgebra.
Denote by $\Delta: M\to M\ovt M$ the comultiplication map
	and $\X$ the $M\ovt M$-boundary piece associated with $\{M\otimes 1, 1\otimes M\}$.
Then there exists a u.c.p.\ map $\phi: \tilde \bS_\X(M\ovt M)\to \tilde \bS(N)$
	such that $\phi(x)=E_N^M\circ \Delta^{-1}\circ E_{\Delta(M)}^{M\ovt M}(x)$ for all $x\in M\ovt M$. 
	
In particular, if there exists an $N$-central state $\varphi:\tilde\bS(N)\to \C$ with $\varphi_{\mid N}=\tau_N$, 
	then $\psi:=\varphi\circ \phi: \tilde \bS_\X(M\ovt M)\to \C$ is $\Delta(N)$-central 
	with $\psi(x)=\tau(x)$ for any $x\in M\ovt M$,
	where $\tau $ is a trace on $M\ovt M$.
\end{cor}
\begin{proof}
Let $\phi: \B(L^2M\otimes L^2 M)\to \B(L^2M)$ be the u.c.p.\ map as above.
By Lemma~\ref{lem: comultiplication}, (\ref{item: image}),
	we have $\phi^*$ maps $\B(L^2M)^\sharp_J$ to $\B(L^2(M\ovt M))^\sharp_J$
	and hence by taking the bidual of $\phi$ we obtain
	$\tilde \phi: (\B(L^2(M\ovt M))^\sharp_J)^* \to (\B(L^2M)^\sharp_J)^*$,
	which satisfies $\tilde \phi (\iota_{\rm nor}(x))=\iota_{\rm nor}(\phi(x))$
	for any $x\in M\ovt M$ and $J(M\ovt M)J$.
Moreover, from Lemma~\ref{lem: compact comultiplication}, we have $\tilde \phi: \K_\X(M\ovt M)\to \K^{\infty,1}(M)\subset (\K(M)^\sharp_J)^*$.
By continuity of $\tilde \phi$, it follows that $\tilde \phi: (\K_\X(M\ovt M)^\sharp_J)^*\to (\K(M)^\sharp_J)^*$.
Thus we have 
$\tilde \phi: \tilde \bS_\X(M\ovt M)\to \tilde\bS(M)$
	with $\tilde\phi(x)=(\Delta^{-1}\circ E^{M\ovt M}_{\Delta(M)})(x)$ for any $x\in M\ovt M$.
Combining $\tilde \phi$ with \cite[Lemma 3.2]{Din22} yields a u.c.p.\ map
	$\psi: \tilde \bS_\X(M\ovt M)\to \tilde \bS(N)$ with the desired property.
\end{proof}

%\textcolor{blue}{Suppose $L^\infty(X^\Gamma)\rtimes\Gamma\cong L^\infty(Y)\rtimes\Lambda$.
%	Is it true that for any sequence $u_n\in \cU(L\Gamma)$ converging to $0$ weakly,
%		one has $\|E_{L^\infty(Y)}(u_n)\|_2\to 0$?
%	It can be shown that $\limsup_n\|E_{L^\infty(Y)}(u_n)\|_2<1$.
%	Maybe this is enough if one does the analysis on the bidual?}

\begin{lem}\label{lem: comultiplication unitarily conjugate}
Let $\Lambda$ be an i.c.c.\ countable discrete group, $M=L\Lambda$
	and $N,P \subset M$  von Neumann subalgebras with $\cZ(N)=\C$.
Denote by $\Delta: M\to M\ovt M$ the comultiplication map given by $M= L\Lambda$.

Suppose there exist some nonzero projection $p\in \Delta(N)'\cap (M\ovt M)$ 
	and some partial isometry $v\in M\ovt M$ such that $v^* (\Delta(N)p) v\subset P\ovt M$ and $vv^*=p$.

Then there exists a c.c.p.\ map $\phi:\B(L^2(P\ovt M))\to \B(L^2N)$ such that

\begin{enumerate}
	\item For any $x\in P\ovt M$ 
		we have $\phi(x)=(\Delta^{-1}\circ E_{\Delta(N)}^{M\ovt M})(pvxv^*p)$.
	\item For any $y\in N$, we have $\phi(J(v^* \Delta(y) p v)J)=\tau(p) JyJ$,
			where $\tau$ is the trace on $M\ovt M$.
	\item Denoting by $\X_{P\otimes 1}$ the $P\ovt M$ boundary piece associated with $\{P\otimes 1\}$, we have $\phi$ maps $\K_{\X_{P\otimes 1}}^{\infty,1}(P\ovt M)$ to $\K^{\infty,1}(N)$.
\end{enumerate}
%Then $v^* (\Delta(N)p)v$ is not properly proximal relative to $P\ovt 1$ in $P\ovt M$.
\end{lem}

\begin{proof}
Setting $q=v^*v$, we have $pv=vq$ and hence $v^*(\Delta(N)p )v= (v^*\Delta(N) v )q$ and  $[q, v^* \Delta(N) v]=0$.
We may extend $v$ to a unitary $u\in M\ovt M$, which satisfies $uq=vq$ and $pu=pv$,
	and it follows that $v^*(\Delta(N)p)v=(u^*\Delta(N) u )q$ with $[q, u^*\Delta(N) u]=0$.

Put $N_1=u^*\Delta(N)u$ and $\tau$ to be the canonical trace on $M\ovt M$.
%Denote by $\tau_q(\cdot)=\tau(\cdot)/\tau(q)$ the trace on $q(M\ovt M)q$.

Since $N$ is a factor, the map $V_0: L^2(N_1,\tau)\ni \hat x\to q\hat x\in L^2(N_1 q, \tau)$ has norm $\tau(q)^{1/2}$ 
	and $\Ad(V_0): \B(L^2(N_1 q, \tau))\to \B(L^2(N_1, \tau))$
	satisfies $\Ad(V_0)(qx)=\tau(q) x$ and $\Ad(V_0)(J qxJ)=\tau(q)JxJ$ for any $x\in N_1$.
	
Consider
\[
\begin{aligned}
V: L^2N\xrightarrow{\Delta} L^2(\Delta(N),\tau)\xrightarrow{u^*Ju^*J} L^2(N_1,\tau)\xrightarrow{V_0}
L^2(N_1q,\tau)\xrightarrow{\iota} L^2( P\ovt M ,\tau),
\end{aligned}
\]
where $\iota: L^2(N_1 q,\tau)\to L^2(P\ovt M,\tau)$ is the inclusion map,
and set 
$$\phi:=\Ad(V):\B(L^2(P\ovt M,\tau))\to \B(L^2N).$$

For any $a\in N$,
one checks that $V\hat a=u^*\Delta(a)u\hat q$
	and thus $\phi(x)=\Delta^{-1}\circ E_{\Delta(N)}(u qxq u^*)$ for $x\in P\ovt M$,
	and $\phi(J(u^* \Delta (y) u q)J)=\tau(p)JyJ$ for any $y\in N$.
In particular, $\tau(p)^{-1}\phi$ is a u.c.p.\ map.
	
Now we show $\phi(\K_{\X_{P\ovt 1}}^{\infty,1}(P\ovt M))\subset \K^{\infty,1}(N)$.
Indeed, from Lemma \ref{lem: comultiplication}, (\ref{item: compact image}) and (\ref{item: infty-1 continuity}),  
	we see that if $a_i\in (N)_1$ is a sequence converging to $0$ weakly, 
		then $\|K\Delta(\hat a_i)\|\to 0$ for any $K\in \K_{\X_{M\otimes 1}}(M\ovt M)$.
It follows that 
	$$\|K u^* \Delta(a_i) u\hat q\|=\|K u^* Jqu^*J \Delta(\hat a_i)\|\to 0,$$
	as $i\to \infty$ for any $K\in \K_{\X_{M\otimes 1}}(M\ovt M)$,
	since $\K_{\X_{M\otimes 1}}(M\ovt M)$ is invariant under pre and post composition with $M\ovt M$ and $J(M\ovt M)J$.
Since $\Ad(e_{P\ovt M}): \K_{\X_{M\otimes 1}}^{\infty,1}(M\ovt M)\to \K_{\X_{P\otimes 1}}^{\infty,1}(P\ovt M)$ is onto, we have 
	$$\phi(\K_{\X_{P\otimes 1}}^{\infty,1}(P\ovt M))\subset \K^{\infty,1}(N).$$
%and by $\|\cdot\|_{\infty,1}$ continuity of $\phi$ we further have 
%	$$\phi(\K_{\X_{P\ovt 1}}^{\infty,1}(P\ovt M))\subset \K^{\infty,1}(N).$$
%
%Therefore, for any $T\in \bS_{\X_{P\ovt 1}}(P\ovt M)$, 
%	we have $[\phi(T), JxJ]=\tau(p)^{-1} \phi([T, Ju^* \Delta(x) u q J])\in \K^{\infty,1}(N)$,
%	i.e., $\phi(T)\in \bS(N)$.
%Denote by $\varphi:\bS(N)\to \C$ an $N$-central state with $\varphi_{\mid N}=\tau_N$
%	and set $\psi:=\tau(q)^{-1}\varphi\circ \phi:\bS_{\X_{P\otimes 1}}(P\ovt M)\to \C$.
%Finally one verifies that $\psi$ is a $(u^*\Delta(N)u)q$-central state such that $\psi_{\mid q(P\ovt M)q}=\tau$.
\end{proof}

\begin{cor}\label{cor: non prop prox comultiplication 2}
Let $\Lambda$ be an i.c.c.\ countable discrete group, $P\subset M:=L\Lambda$ a von Neumann subalgebra and $N\subset M$ a subfactor.
Denote by $\Delta: M\to M\ovt M$ the comultiplication map.

Suppose there exist a nonzero projection $p\in \Delta(N)'\cap (M\ovt M)$ and a partial isometry $v\in M\ovt M$ such that 
	$v^*(\Delta(N)p)v\subset M\ovt P$ and $vv^*=p$.
Then there exists a c.c.p.\ map $\phi: \tilde \bS_{\X_{P\otimes 1}}(P\ovt M)\to \tilde \bS(N)$
	such that  $\phi(\iota_{\rm nor}(x))=(\iota_{\rm nor}\circ \Delta^{-1}\circ E_{\Delta(N)}^{M\ovt M})(pvxv^*p)$
	for any $x\in P\ovt M$.

In particular, if there exists an $N$-central state $\varphi:\tilde\bS(N)\to \C$ with $\varphi_{\mid N}=\tau$, 
	then $\psi:=\varphi\circ \phi: \tilde \bS_{\X_{P\otimes 1}}(P\ovt M)\to \C$ is $v^*(\Delta(N)p)v$-central 
	and $\psi(x)=\tau(qxq)$ for $x\in P\ovt M$, where $q=v^*v$.
\end{cor}
\begin{proof}
This follows from Lemma~\ref{lem: comultiplication unitarily conjugate}, similar to the argument in Corollary~\ref{cor: non prop prox comultiplication 1}.
\end{proof}

\section{Proof of main theorems}

Now we are ready to demonstrate rigidity of non properly proximal groups in the context of Bernoulli actions.

\begin{prop}\label{prop: conjugate Delta}
Let $\Gamma$ be a countable discrete group that is i.c.c.\ and nonamenable.
Suppose $\Gamma$ is non properly proximal with $\Lambda_{\rm cb}(\Gamma)=1$, and $L(\Z\wr\Gamma)\cong L\Lambda$ for some group $\Lambda$.
Then there exists a unitary $u\in L\Lambda$ such that $u^*\Delta(L\Gamma) u\subset L\Gamma\ovt L\Gamma$, 
	where $\Delta: L\Lambda\to L\Lambda\ovt L\Lambda$ is the comultiplication map.

\end{prop}
%
%\begin{thm}\label{thm: recover Bernoulli}
%Let $\Gamma$ be a countable discrete group that is i.c.c.\ and nonamenable.
%Suppose $\Gamma$ is non properly proximal with $\Lambda_{\rm cb}(\Gamma)=1$.
%If $L(\Z\wr\Gamma)\cong L\Lambda$ for some group $\Lambda$,
%then there exists an infinite abelian group $H$ and an action $\Gamma\actson  H$ by automorphisms such that $\Lambda=H\rtimes\Gamma$ and $\Gamma\actson [0,1]^\Gamma$ is conjugate to $\Gamma\actson \hat H$.
%\end{thm}
\begin{proof}
Set $M=L(\Z\wr\Gamma)$ and $N=\Delta(L\Gamma)$.
%and $\Delta: M\to M\ovt M$ the comultiplication map given by $M\cong L\Lambda$
%	and set 
Since $(\Z\wr\Gamma) \times (\Z\wr\Gamma)$ is biexact relative to 
	$\{\Gamma\times (\Z\wr\Gamma), (\Z\wr\Gamma)\times \Gamma\}$ \cite[Chapter 15]{BrOz08},
	and $N$ has no amenable direct summand,
	we have $N$ is properly proximal relative to $\{L\Gamma\ovt M, M\ovt L\Gamma\}$ in $M\ovt M$ as well
	by \cite{DP22}.

Furthermore, as $\Gamma$ is non properly proximal, we have a $L\Gamma$-central state $\varphi_0: \tilde\bS(L\Gamma)\to \C$ 
	such that $\varphi_{\mid L\Gamma}=\tau$,
	which yields an $N$-central state $\varphi: \tilde \bS_\X(M\ovt M)\to \C$ with $\varphi_{\mid M\ovt M}=\tau$ by Corollary~\ref{cor: non prop prox comultiplication 1},
	where $\X$ denotes the $M\ovt M$-boundary piece associated with $\{M\otimes 1, 1\otimes M\}$.

Observe that $\Gamma<\Z\wr\Gamma$ is almost malnormal and thus $\Gamma\times (\Z^\Gamma\rtimes \Gamma)$
	(resp.\ $(\Z^\Gamma\rtimes \Gamma)\times \Gamma)$ 
	is almost malnormal relative to $\{e\}\times (\Z^\Gamma\rtimes \Gamma)$
	(resp. \ $(\Z^\Gamma\rtimes \Gamma)\times \{e\})$ in $(\Z\wr\Gamma)\times(\Z\wr\Gamma)$.
%Furthermore, as $\Gamma$ is non properly proximal, we have $N\subset M\ovt M$ is non properly proximal relative to $\{M\otimes 1, 1\otimes M\}$ by Lemma~\ref{lem: non prop prox and comultiplication}.
Now we are in the situation which Proposition~\ref{prop: non prop prox to rel amen} applies to, 
	with $\Lambda_1= \Gamma\times (\Z\wr\Gamma), \Lambda_2= (\Z\wr\Gamma)\times \Gamma$,
	and $\Sigma_1=\{e\}\times (\Z\wr\Gamma), \Sigma_2=(\Z\wr\Gamma)\times \{e\}$,
	and it follows that there exists a partition of unity $p_1,p_2\in \cZ(N'\cap (M\ovt M))$
	such that $Np_1$ is amenable relative to $L\Gamma\ovt M$ in $M\ovt M$,
	and $Np_2$ is amenable relative to $M\ovt L\Gamma$ in $M\ovt M$.

Since $\Gamma$ is nonamenable, we have $N$ is strongly nonamenable relative to $1\otimes M$ and $M\otimes 1$,
	$N\not\prec_{M\ovt M} L(\oplus_\Gamma\Z)\ovt M$ and $N\not\prec_{M\ovt M}M\ovt L(\oplus_\Gamma\Z)$ by \cite[Proposition 7.2]{IoPoVa13}.
It then follows from Proposition~\ref{prop: intertwining from rel amen} that we may find partial isometries $v_1, v_2\in M\ovt M$ such that 
	$v_1^*(Np_1) v_1\subset M\ovt L\Gamma$ with $v_1 v_1^*=p_1$, 
	and $v_2^*(Np_2) v_2\subset  L\Gamma \ovt M$ with $v_2 v_2^* =p_2$.

We then consider $N_1:=v_1^*(Np_1)v_1$.
One checks that since $N$ is strongly nonamenable relative to $1\otimes M$ in $M\ovt M$,
	we have $N_1$ is also strongly nonamenable relative to $1\otimes L\Gamma$ in $M\ovt L\Gamma$.	
Thus by Proposition~\ref{prop: dichotomy}, one has $N_1$ is properly proximal relative to $L\Gamma\ovt L\Gamma$.

By Corollary~\ref{cor: non prop prox comultiplication 2}, 
	we have $\psi:\tilde \bS_{\X_{1\otimes L\Gamma}}(M\ovt L\Gamma)\to \C$ is an $N_1$-central state 
	with $\psi(q_1x q_1)=\tau(p_1)^{-1} \tau(q_1xq_1)$ for $x\in M\ovt L\Gamma$,
	where $q_1=v_1^* v_1$.
%Moreover,
%	denote by $\phi: \tilde \bS_{\X_{L\Gamma\otimes 1}}(L\Gamma\ovt M)\to \tilde \bS(L\Gamma)$ the c.c.p.\ map from Lemma 5.?.
%Setting $\psi:=\varphi_0\circ \phi$, we have $\psi:\tilde \bS_{\X_{L\Gamma\otimes 1}}(L\Gamma\ovt M)\to \C$ is an $N_1$-central state with $\psi_{\mid q_1(L\Gamma\ovt M)q_1}=\tau(p_1)\tau$,
%	where $N_1=v_1^*(Np_1)v_1$ and $q_1=v_1^* v_1$.
We may then apply Proposition~\ref{prop: non prop prox to rel amen},
	which yields that $N_1$ is amenable relative to $L\Gamma\ovt L\Gamma$ in $M\ovt L\Gamma$.
		
As $N_1\not\prec_{M\ovt L\Gamma} L(\oplus_\Gamma\Z)\ovt L\Gamma$,
by invoking Proposition~\ref{prop: intertwining from rel amen} one more time, we obtain a partial isometry $w_1\in M\ovt L\Gamma$ 
	such that $w_1^* N_1 w_1\subset L\Gamma\ovt L\Gamma$ and $w_1 w_1^*=q_1$.

Similarly, we obtain a partial isometry $w_2\in L\Gamma\ovt M$ such that $w_2^*(v_2^* Np_2 v_2)w_2\subset L\Gamma\ovt L\Gamma$ and
	$w_2w_2^*= v_2 v_2^*$.
	
Finally, observe that $\Ad(v_2 w_2)(p_2)\in L\Gamma\ovt L\Gamma$ and $1-\Ad(v_1 w_1)(p_1)\in L\Gamma\ovt L\Gamma$
	are equivalent as $p_1+p_2=1$.
Thus we may find a partial isometry $w_2'\in L\Gamma\ovt \Gamma$ that implements their equivalence.
Set $u=v_1w_1+ v_2w_2 w_2'$ and one verifies that $u$ is a unitary and $u^*N u\subset L\Gamma\ovt L\Gamma$.
\end{proof}

\begin{proof}[Proof of Theorem~\ref{thm: recover Bernoulli}]
This follows directly from the proof of \cite[Theorem 8.2]{IoPoVa13}, as its step 1 is established in Proposition~\ref{prop: conjugate Delta}.
\end{proof}

The following elementary lemma may be derived directly from \cite[Theorem 6.2]{Ioa11}.
We nevertheless give a straightforward proof. 
\begin{lem}\label{lem: normal abelian subgroup}
Let $H$ be an abelian group and $\Gamma$ be an i.c.c.\ group that acts on $H$ by automorphisms.
Suppose $\Gamma\actson \widehat H$ is conjugate to the Bernoulli action $\Gamma\actson [0,1]^\Gamma$.
Then for any nontrivial normal abelian subgroup $K\lhd H\rtimes\Gamma$, we have $K<H$.
\end{lem}
\begin{proof}
First we show that $K\cap H$ is nontrivial.
Indeed, from $K\cap H=\{e\}$ one sees that $K$ is in the centralizer of $H$ in $H\rtimes \Gamma$.
As $LH\subset L(H\rtimes\Gamma)$ is maximal abelian, we have $LK\subset LH$, contradicting $K\cap H=\{e\}$.

We denote by $E_S$ the conditional expectation from $\ovt_{\Gamma}L^\infty([0,1])$
	to $\ovt_S L^\infty([0,1])$ for a subset $S\subset \Gamma$.
For any function $f\in \ovt _\Gamma L^\infty([0,1])$ and any $\varepsilon>0$, one may consider
$$
F_{f,\varepsilon}:=\{t\in \Gamma\mid \|P_t(f)-f\|_2>\varepsilon\},
$$
where $P_t=E_{\Gamma\setminus\{t\}}$.
%where $P_t: \ovt_\Gamma L^\infty([0,1])\to \ovt_{\Gamma\setminus\{t\}} L^\infty([0,1])$ is the map that equals to the trace on the $t$-th component and identity everywhere else.
We show that for a non-scalar $f$, we may find some $\varepsilon$ such that $F_{f,\varepsilon}$ is a nonempty finite set.

If $F_{f,\varepsilon}$ was empty for any $\varepsilon>0$,
	then one would have $P_t(f)=f$ for any $t\in \Gamma$
	and hence $(\prod_{t\in S} P_t)(f)=f$ for any finite $S\subset \Gamma$.
For any $\varepsilon>0$, take $S\subset \Gamma$ a finite set such that 
	$\|E_S(f)-f\|_2<\varepsilon$.
Then
$\|\tau(f)-f\|_2=\|(\prod_{t\in S} P_s)(E_S(f)-f)\|_2<\varepsilon$
	and hence $f=\tau(f)$.
	
To see $F_{f,\varepsilon}$ is finite for any $\varepsilon>0$,
%since ${\rm sp}\{g\mid g\in \ovt_F L^\infty([0,1]), F\subset \Gamma{\rm \ finite\ subset}\}$ 
%is $\|\cdot\|_2$-dense in $\ovt_\Gamma L^{\infty}([0,1])$ 
%	and $\pi_t$ is continuous in $\|\cdot\|_2$, 
%	the set $F_{f,\varepsilon}$ is finite.
we may find some finite subset $S\subset \Gamma$ such that $\|f-E_{S}(f)\|_2<\varepsilon/ 2$.
Notice that for any $t\notin S$, we have 
$P_t(E_{S}(f))=E_{S}(f)$ and hence
$P_t(f)-f=P_t(f-E_{S}(f))+ E_{S}(f))-f$.
%
% -E_{F_\varepsilon}(f))=\pi_t(f)-\tau(E_{F_\varepsilon}(f))=\pi_t(f)-\tau(f)$.
As $P_t$ is $\|\cdot\|_2$-continuous, we have 
	$\|P_t(f)-f\|_2\leq 2\|f-E_{S}(f))\|_2<\varepsilon$ for $t\notin S$.

%Moreover, if $F_{f,\varepsilon}$ were empty for any $\varepsilon>0$,
%	then one would have $E_F(f)=\tau(f)$ for any subset $F\subset \Gamma$,
%	where $E_F: \ovt_\Gamma L^\infty([0,1])\to \ovt_F L^\infty([0,1])$ is the conditional expectation, 
%	and hence $f=\tau(f)$.

Since $K\cap H$ is nontrivial, we may take some nontrivial $g\in K\cap H$ and view $\lambda_g\in LH\cong L^\infty([0,1])^\Gamma$ as a function that is not a scalar
	as $\tau_{L(H\rtimes\Gamma)}(\lambda_g)=0$.
Thus we may find some $\varepsilon>0$ such that $F_g:=F_{\lambda_g, \varepsilon}\subset \Gamma$ is a finite nonempty set.
Note that if $s\in \Gamma$ fixes $g$, then $sF_g=F_g$ as $\pi_t(\sigma_s(f))= \pi_{s^{-1}t}(f)$ for any $t\in \Gamma$,
	where we denote by $\sigma$ the Bernoulli action $\Gamma\actson [0,1]^\Gamma$.

For any $hs\in K$ with $h\in H$ and $s\in \Gamma$, since $K$ is abelian and normal, we have $[hs, t g t^{-1}]=e$ for any $t\in \Gamma$.
In particular, this implies $t^{-1} s tF_g=F_g$ for any $t\in \Gamma$, which in turn shows that $\{t^{-1} s t\mid t\in \Gamma\}$ is finite,
	as $F_g\subset \Gamma$ is a nonempty finite set.
As $\Gamma$ is i.c.c., we conclude $s$ must be trivial and hence $K<H$.
\end{proof}

\begin{proof}[Proof of Theorem~\ref{thm: semidirect product CMAP}]
Suppose $\Gamma$ is in addition properly proximal, this follows from \cite[Theorem 1.5]{BoIoPe21}.

If $\Gamma$ is non properly proximal, then by Theorem~\ref{thm: recover Bernoulli} 
	one obtains the conclusion of \cite[Theorem 8.2]{IoPoVa13}:
	there exists a group isomorphism $\delta: \widehat Y\rtimes\Lambda\to  H\rtimes \Gamma$ for some abelian group $H$ and $\Gamma\actson^\alpha H$ by automorphisms,
	a $*$-isomorphism $\theta: L(H)\to L(\oplus_\Gamma\Z)$ satisfying $\theta\circ \alpha_g=\sigma_g\circ \theta$ for all $g\in\Gamma$,
	a character $\eta: \Z\wr\Gamma\to \C$ and a unitary $w\in L(\Z\wr\Gamma)$
	such that the isomorphism $\pi: L(\widehat Y\rtimes\Lambda)=L(\Z\wr\Gamma)$ is given by 
	$\pi=\Ad(w)\circ \pi_\eta \circ \pi_\theta \circ \pi_\delta$,
	with 
 	$$\pi_\delta: L(\widehat Y\rtimes\Lambda)\ni \lambda_g\to \lambda_{\delta(g)}\in L(H\rtimes\Gamma),$$
 	$$\pi_\theta: L(H)\rtimes\Gamma\ni au_g\mapsto \theta(a)u_g\in L(\oplus_\Gamma\Z)\rtimes\Gamma,$$
 	$$\pi_\eta: L(\Z\wr\Gamma)\ni \lambda_g\to \eta(g)\lambda_g\in L(\Z\wr\Gamma).$$
 	
By Lemma~\ref{lem: normal abelian subgroup}, we have $\delta(\widehat Y)<H$ and hence 
$\Ad(w^*)\circ \pi(L^\infty(Y))=L(\oplus_\Gamma\Z)$.
%Note that if $ht\in \Sigma$ for $h\in H$ and $t\in \Gamma$, then $t$ must be in the center of $\Gamma$, which is trivial and thus $\Sigma< H$.
%
%If $\Sigma\cap H$ is nontrivial, then $\Sigma$ is contained in $H$. Indeed, suppose $e\neq g\in \Sigma\cap H$ and $ht\in (H\rtimes\Gamma)\cap \Sigma$, where $h\in H $ and $t\in \Gamma\setminus\{e\}$.
%Notice that $ht g= \sigma_t(g) h t$ and thus if $[g, ht]=0$ we would then have $\sigma_t(g)=g$,
%	contradicting the fact that $\Gamma\actson \hat H$ is free.
%
%
%Since $L\Gamma\subset L(H\rtimes \Gamma)$ is isomorphic to $L\Gamma\subset L(\Z\wr\Gamma)$
%	and $\Sigma$ is normal in $\Sigma\rtimes\Lambda$,
%	we have $\Sigma\cap \Gamma$ must be finite.
\end{proof}

\bibliographystyle{amsalpha}
\bibliography{ref}

\end{document}